\documentclass[12pt,reqno,a4paper]{amsart}

\usepackage{mdframed}
\usepackage{setspace}
\usepackage{mathrsfs}

\usepackage[latin1]{inputenc}

\usepackage[linkcolor=red,colorlinks=true]{hyperref}

\usepackage{etoolbox}
\usepackage{amsmath}
\usepackage{enumerate}
\usepackage{amssymb}
\usepackage{amscd}
\usepackage{amsthm}
\usepackage{amsfonts}
\usepackage{graphicx}
\usepackage[all,cmtip]{xy}
\usepackage{enumitem}

\patchcmd{\subsection}{-.5em}{.5em}{}{}
\patchcmd{\subsubsection}{-.5em}{.5em}{}{}

\usepackage{enumitem}

\makeatother
\usepackage{hyperref}

\numberwithin{equation}{section}

\newcommand{\SL}{\operatorname{SL}}

\newcommand{\cA}{\mathcal{A}}

\newcommand{\cC}{\mathcal{C}}

\newcommand{\cS}{\mathcal{S}}

\newcommand{\cW}{\mathcal{W}}

\newcommand{\bC}{\mathbb{C}}

\newcommand{\bN}{\mathbb{N}}

\newcommand{\bQ}{\mathbb{Q}}
\newcommand{\bR}{\mathbb{R}}

\newcommand{\bZ}{\mathbb{Z}}

\newcommand{\gog}{\mathfrak{g}}

\newcommand{\gol}{\mathfrak{l}}

\newcommand{\gos}{\mathfrak{s}}

\newcommand{\ra}{\rightarrow}

\newcommand{\qand}{\quad \textrm{and} \quad}

\newcommand\subsetsim{\mathrel{%
\ooalign{\raise0.2ex\hbox{$\subset$}\cr\hidewidth\raise-0.8ex\hbox{\scalebox{0.9}{$\sim$}}\hidewidth\cr}}}
\newcommand{\eps}{\varepsilon}

\DeclareMathOperator{\dist}{dist}

\DeclareMathOperator{\Lip}{Lip}
\DeclareMathOperator{\supp}{supp}

\DeclareMathOperator{\Diag}{Diag}

\DeclareMathOperator{\Mat}{Mat}

\DeclareMathOperator{\id}{id}

\DeclareMathOperator{\Poi}{Poi}

\DeclareMathOperator{\Wei}{Wei}

\DeclareMathOperator{\Vol}{Vol}

\newcommand{\bd}{\underline{d}}

\definecolor{lichtgrijs}{gray}{0.95}
\mdfdefinestyle{mystyle}{ %
    backgroundcolor=lichtgrijs, %
    linewidth=0pt, %
    innertopmargin=10pt, %
    innerbottommargin=10pt,%
    nobreak=true
}

\theoremstyle{theorem}
\newtheorem{theorem}{Theorem}[section]

\newtheorem{proposition}[theorem]{Proposition}
\newtheorem{problem}[theorem]{Problem}
\newtheorem{example}[theorem]{Example}
\newtheorem{lemma}[theorem]{Lemma}

\theoremstyle{definition}
\newtheorem{definition}[theorem]{Definition}
\newtheorem{remark}[theorem]{Remark}

\usepackage{changepage}
\usepackage{mathtools}
\usepackage{graphicx}
\usepackage{calc}

\usepackage[toc,page]{appendix}

\usepackage{scalerel}

\usepackage{extarrows}

\begin{document}
\bibliographystyle{plain} 

\title[Poisson and Weibull asymptotics]{Poisson approximation and Weibull asymptotics in the geometry of numbers}

\author{Michael Bj\"orklund}
\address{Mathematical Sciences,  Chalmers,  Chalmers Tv\"argata 3, 412 58 Gothenburg,  Sweden}
\email{micbjo@chalmers.se}

\author{Alexander Gorodnik}
\address{Department of Mathematics,  University of Z\"urich,  Winterthurerstrasse 190,  CH-8057,  Switzerland}
\email{alexander.gorodnik@math.uzh.ch}
\thanks{MB was supported by GoCas Young Excellence grant 11423310 and Swedish VR-grant 11253320,  AG was supported by the SNF grant 200021--182089.}



\begin{abstract}
Minkowski's First Theorem and Dirichlet's Approximation Theorem provide upper bounds on certain minima taken over lattice points contained in domains of Euclidean spaces.
We study the distribution of such minima and show,  under some technical conditions,  that they exhibit 
Weibull asymptotics with respect to different natural measures on the space of unimodular lattices in 
$\bR^d$.  This follows from very general Poisson approximation results for shrinking targets which should be of independent interest. 
Furthermore,  we show in the appendix that the logarithm laws of Kleinbock-Margulis,  Khinchin and Gallagher can be deduced from our distributional results.
\end{abstract}

\maketitle

\section{Introduction}

In this paper we consider several extremal problems in Geometry of Numbers such as the existence of lattice points in convex bodies,  minimal values of polynomial maps and 
Diophantine approximation.  Our first results deal with generic unimodular lattices,  for which we use the following notation:
\begin{align*}
X_d &= \textrm{the space of lattices in $\bR^d$ with covolume one},  \\
\mu_d &= \textrm{the unique $\SL_d(\bR)$-invariant probability measure on $X_d$. }
\end{align*}
We write $\Vol_d$ for the Lebesgue measure on $\bR^d$,  normalized so that $\Vol_d\big([0,1]^d\big) = 1$.

\subsection {Application I: Minkowski's Theorem}

Let $B \subset \bR^d$ be a compact, convex, and symmetric set with non-empty interior and let 
$\Lambda$ be a lattice in $\bR^d$ of covolume one.  The Minkowski's First Theorem (\cite[Lecture II,  Theorem 11]{SiegelBook}) says that
\[
\Vol_d(B) \geq 2^d \implies \Lambda\cap B\ne 0.
\] 
In order to discuss whether this result can be improved,  it is natural to introduce the function
\begin{equation}\label{eq:cB}
c_B(\Lambda) := \inf\{ c > 0 \,  : \,  \Lambda \cap cB \neq 0 \big\},  \quad \textrm{for $\Lambda \in X_d$}.  
\end{equation}
Let
\begin{equation}
\label{eq:T}
\mathscr{T}_{d}:=\big\{ (T_1,\ldots,T_{d})\in \bR_+^d:\, T_1\cdots T_d=1\big\},
\end{equation}
and for $T\in \mathscr{T}_{d}$, 
we consider the boxes
\[
B_T := \prod_{p=1}^d [-T_p,T_p] \subset \bR^d.
\]
Since $\hbox{vol}(B_T)=2^d$, it follows from Minkowski's First Theorem that 
$c_{B_T}(\Lambda) \leq 1.$
Furthermore,  it is easy to see that 
\begin{equation}
\label{eq:Mink0}
\sup \big\{ c_{B_T}(\Lambda) \,  : \,  \Lambda \in X_d \big\} = 1,
\end{equation}
whence Minkowski's First Theorem is in some sense optimal.  However,  we can significantly improve it
for generic lattices:

\begin{theorem}\label{th:l1}
For $\mu_d$-almost every $\Lambda\in X_d$,
\begin{equation}
\label{eq:Mink}
\limsup_{\|T\|_\infty\to\infty}\; -\frac{\log c_{B_T}(\Lambda) }{\log\log \|T\|_\infty}\ge \frac{d-1}{d}.
\end{equation}
\end{theorem}

\noindent This result can be proved using,  for instance,  the method introduced in \cite{KM}.

\begin{remark}
Theorem \ref{th:l1} is an example of a \emph{logarithm law}.  Such laws have been 
extensively studied in Number Theory and Dynamics.
Perhaps the first occurrence of a logarithm law is Khinchin's Theorem (cf. \eqref{DS}).  It was realized by Sullivan \cite{sul} that Khinchin's Theorem is intimately connected with geodesic excursions 
to shrinking cuspidal neighborhoods in finite volume hyperbolic orbifolds.  Later  this idea was generalized 
and refined in the influential work of Kleinbock and Margulis \cite{KM}.
\end{remark}

It turns out that the logarithm law \eqref{eq:Mink},  as well as other 
logarithm laws that we will discuss later,  are very much related to \emph{distributional} convergence of
extremal values of arithmetic functions.  We show in the appendix (see Proposition \ref{p:log1}) that
Theorem \ref{th:l1} is a direct corollary of the following result:

\begin{theorem}
	\label{Th_M}
For a sequence of finite subsets $\Delta_n$ of $\mathscr{T}_{d}$
such that 
\begin{equation}\label{eq:condd}
|\Delta_n|\to\infty \quad\quad\hbox{and}\quad\quad  \frac{\displaystyle \min_{T\ne T'\in \Delta_n} \max_{p} |\log T_p - \log T'_p|}{\log |\Delta_n|} \to \infty,
\end{equation}
the minima
	\[
	\mathfrak{C}_{\Delta_n}(\Lambda) := \min_{T\in\Delta_n} 
	c_{B_{T}}(\Lambda)
	\]
satisfy
	\[
	|\Delta_n|^{1/d} \cdot \mathfrak{C}_{\Delta_n} \xLongrightarrow[\mu_d]{}
	\Wei\big(2^{-(d-1)/d} \zeta(d)^{1/d},d\big),  
	\]
where $\Wei$ denotes the Weibull distribution.
\end{theorem}

\begin{remark}
Here and throughout this paper,  $\zeta$ denotes Riemann's $\zeta$-function. 
We also recall that the \emph{Weibull distribution} $\Wei(\lambda,a)$ 
with scale $\lambda$ and shape $a$
is defined by the distribution function $1-\exp({-(\lambda t)^a})$.
\end{remark}

\vspace{0.2cm}

\subsection{Application II: Minima of polynomial maps}

We now turn to the very classical study of minimal values of homogeneous forms
restricted to lattice points (see \cite[Ch.~6]{GL} and \cite[Chapter IX]{SiegelBook}).  
Unlike Theorem \ref{Th_M},  this involves finding lattice points in \emph{non-convex} sets.  
We consider the product form
$$
\Pi(x):=\prod_{p=1}^d |x_p|,  \quad \textrm{for $x \in \bR^d$}.
$$
The following result was proved by Kleinbock and Margulis (\cite[Subsection 1.11]{KM}):
\begin{theorem}[Kleinbock--Margulis]
\label{th:l2}
For $\mu_d$-almost every  $\Lambda \in X_d$,
\begin{equation}
\label{eq:prod}
\limsup_{v\in\Lambda: \,  \|v\|_\infty\to\infty}\; -\frac{\log \Pi(v) }{\log\log \|v\|_\infty}\ge d-1.
\end{equation}
\end{theorem}

We establish a distributional version of this result. For boxes $B_T$ with $T\in\mathscr{T}_d$ and a lattice $\Lambda$, we define 
$$
m_{B_T}(\Lambda):=\inf\big\{ \Pi(v):\, 0\ne v\in \Lambda\cap B_T \big\}.
$$
With this notation, we show:

\begin{theorem}
	\label{Th_PLF}
Let $\Delta_n$ be a sequence of finite subsets of $\mathscr{T}_d$
satisfying \eqref{eq:condd}, and
\[
\mathfrak{M}_{\Delta_n}(\Lambda) := \min_{T\in \Delta_n} 
m_{B_{T}}(\Lambda).
\]	
Then
	\[
	|\Delta_n|(\log |\Delta_n|)^{d-1}\cdot \mathfrak{M}_{\Delta_n} \xLongrightarrow[\mu_d]{}
	\Wei\big(2^{-(d-1)}\zeta(d)/c_d,1\big),  
	\]
where $c_d$ is the volume of the simplex $\{ u \in [0,1]^{d-1} \,  : \,  u_1 + \cdots + u_{d-1} \leq 1 \big\}$.
\end{theorem}

\begin{remark}
We stress that Theorem \ref{th:l2} is in fact a direct consequence of Theorem \ref{Th_PLF} (see Proposition \ref{p:log1} in the Appendix).
\end{remark}

This result can be generalized as follows.
Let $F$ be a nontrivial homogeneous polynomial on $\bR^d$. 
For a lattice $\Lambda\in X_d$, we consider the family of lattices
$$
\Lambda_T:=\{(T_1\lambda_1,\ldots, T_d\lambda_d):\, \lambda\in\Lambda \}, \quad T=(T_1,\ldots,T_d)\in \mathscr{T}_d.
$$
Extremal behaviour for this family of lattices was studied in \cite[Ch.~3]{GL}.
We investigate the asymptotics of the minima
$$
m_T(F,\Lambda):=\min\left\{ |F(v)|:\, 0\ne v\in \Lambda_T\cap [-1,1]^d \right\}.
$$

\begin{theorem}
	\label{Th_Pol}
	There exist $c>0$, $a\in \bQ^+$, $b\in\bN_0$ such that 
	for any sequence of finite subsets $\Delta_n$ of $\mathscr{T}_d$ satisfying \eqref{eq:condd}, and
	\[
	\mathfrak{M}_{\Delta_n}(F,\Lambda) := \min_{T\in \Delta_n} 
	m_T(F,\Lambda),
	\]	
	one has
	\[
	|\Delta_n|^a(\log |\Delta_n|)^{b}\cdot \mathfrak{M}_{\Delta_n}(F,\cdot ) \xLongrightarrow[\mu_d]{}
	\Wei\big(c,a\big).
	\]
\end{theorem}

\vspace{0.2cm}

\subsection{Application III: Khinchin's Theorem}

Given $\alpha \in \bR^{d-1}$,  we define
\[
k_T(\alpha):=\min\big\{ T^{1/(d-1)} \| q \alpha-p \|_\infty \,  : \, 
p \in \bZ^{d-1},    \enskip 1 \leq q \leq T \big\},  
\quad \textrm{for $T \geq 1$}.
\]
By Dirichlet's Approximation Theorem,  we have $k_T(\alpha) \leq 1$ for all $\alpha \in \bR^{d-1}$. 
Furthermore,  according to Davenport and Schmidt \cite[Theorem 1]{DS},
for Lebesgue almost every $\alpha\in [0,1]^{d-1}$,
\begin{equation}
\label{DS0}
\limsup_{T\to \infty} k_T(\alpha)= 1.
\end{equation}
On the other hand, 
\begin{theorem}[Khinchin] \label{th:l3}
For Lebesgue almost every $\alpha\in [0,1]^{d-1}$,
\begin{equation}
\label{DS}
\limsup_{T\to \infty}\,  -\frac{\log k_T(\alpha)}{\log\log T}\ge -\frac{1}{d-1},  
\end{equation}
\end{theorem}
The limits \eqref{DS0} and \eqref{DS} indicate that the function $T \mapsto k_T(\alpha)$ fluctuates quite wildly,  and we wish to study its extremal behavior.
More precisely,  we  show that minima of the function $T \mapsto k_T(\cdot)$ on sufficiently lacunary subsets of $[1,\infty)$ exhibit non-trivial Weibull asymptotics with respect to the Lebesgue measure on $[0,1]^{d-1}$.  \\

More generally, we can deal with approximation of systems of linear forms.
Let $d\ge 3$ and write $d=d_1+d_2$ with $d_1,d_2\ge 1$. We set
\begin{equation}\label{eq:TT}
\mathscr{T}_{\bd}^{+} := \big\{ T=(T_1,\ldots,T_d) \in \mathscr{T}_{d} \,  : \, T_{1},\ldots, T_{d_1}\ge 1,\, T_{d_1+1},\ldots, T_{d}\le 1  \big\}.
\end{equation}
For $T\in \mathscr{T}_{\bd}^{+}$, we define 
\begin{equation}\label{eq:T}
\lfloor T\rfloor_{\bd}:=\min \big(\log T_{1},\ldots, \log T_{d_1}, -\log T_{d_1+1},\ldots, -\log T_{d}\big).
\end{equation}
Given $T\in\mathscr{T}_{\bd}^{+}$, $\alpha\in \hbox{Mat}_{d_1,d_2}([0,1])$ and $k>0$, 
we consider the inequalities
$$
\Big| p_l+\sum_{j=1}^{d_2} \alpha_{lj} q_j\Big| \le k\, T_l^{-1},\, l=1\ldots,d_1, \quad |q_j|\le T_{d_1+j},\, j=1,\ldots,d_2,
$$
with $p\in \bZ^{d_1}$ and $q\in \bZ^{d_2}$ and define 
$$
k_T(\alpha):=
\min \Big\{
{\max}_l\; T_l \Big|p_l+\sum_{j=1}^{d_2} \alpha_{lj} q_j\Big|\, :\,\,\, p\in\mathbb{Z}^{d_1},\;\; q\in \bZ^{d_2},|q_j|\le T_{{d_1}+j}  \Big\}.
$$

\begin{theorem}
	\label{Th_Dirichlet}
	For a sequence of finite subsets $\Delta_n$ of $\mathscr{T}_{\bd}^{+}$ such that
	\begin{equation}\label{eq:condd2}
	|\Delta_n|\to \infty, \quad
 \frac{\displaystyle \min_{T,T'\in \Delta_n} \max_p |\log T_p - \log T'_p|}{\log |\Delta_n|} \to \infty,\quad 
 \frac{\displaystyle \min_{T\in\Delta_n} \lfloor T\rfloor_{\bd}}{\log |\Delta_n|} \to \infty,
	\end{equation}
	the minima
	\[
	\mathfrak{K}_{\Delta_n}(\alpha): = \min_{T\in\Delta_n} k_{T}(\alpha),  \quad \textrm{for $\alpha \in \hbox{\rm Mat}_{d_1,d_2}([0,1])$},
	\]
satisfy
	\[
	|\Delta_n|^{1/d_1} \cdot \mathfrak{K}_{\Delta_n} \xLongrightarrow[\hbox{\tiny\rm Leb}]{} \Wei\big(2^{-(d-1)/d_1}\zeta(d)^{1/d_1},
	d_1\big). 
	\]
\end{theorem}

\begin{remark}
We again stress that Theorem \ref{th:l3}
follows directly from Theorem \ref{Th_Dirichlet} (see Proposition \ref{p:log2} in the Appendix). 
\end{remark}

\vspace{0.2cm}

\subsection{Application IV: Gallagher's Theorem}
Now we consider a multiplicative version of Khinchin's Theorem.  Given $\alpha \in \bR^{d-1}$,  we define
\[
g_T(\alpha):=\min\left\{ T\cdot {\prod}_{l=1}^{d-1} | q \alpha_l-p_l | \,  : \, 
p \in \bZ^{d-1},    \enskip 1 \leq q \leq T \right\},  
\quad \textrm{for $T \geq 1$}.
\]
The following logarithm law was proved by Gallagher \cite{G}:
\begin{theorem}[Gallagher]\label{th:l4}
For a.e. $\alpha\in \bR^{d-1}$,
\begin{equation}
\label{eq:Gal}
\limsup_{T\to \infty}\,  -\frac{\log g_T(\alpha)}{\log\log T}\ge d-1,  
\end{equation}
\end{theorem}

More generally,  given $T\in\mathscr{T}_{\bd}^{+}$, $\alpha\in \hbox{Mat}_{d_1,d_2}([0,1])$ and $g>0$, 
we consider the inequalities
$$
\prod_{l=1}^{d_1}\Big|p_l+\sum_{j=1}^{d_2} \alpha_{lj} q_j\Big|\le g\, \left({\prod}_{l=1}^{d_1} T_{l}\right)^{-1}, \quad |q_j|\le T_{d_1+j},\, j=1,\ldots,d_2,
$$
with $p\in \bZ^{d_1}$ and $q\in \bZ^{d_2}$ and define 
$$
g_T(\alpha):=
\min \left\{
 \prod_{l=1}^{d_1}T_l\Big|p_l+\sum_{j=1}^{d_2} \alpha_{lj} q_j\Big|\,:\,\,\, 
 \begin{tabular}{ll}
  $q\in \bZ^{d_2},\;|q_j|\le T_{{d_1}+j}$,\\ 
  $p\in\mathbb{Z}^{d_1},\; \Big|p_l+\sum_{j=1}^{d_2} \alpha_{lj} q_j\Big|\le T_l^{-1}$
 \end{tabular}
 \right\}.
$$

\begin{theorem}
	\label{Th_Gal}
	For a sequence of finite subsets $\Delta_n$ of $\mathscr{T}_{\bd}^{+}$ satisfying
	\eqref{eq:condd2}, 	the minima
	\[
	\mathfrak{G}_{\Delta_n}(\alpha): = \min_{T\in\Delta_n} g_{T}(\alpha),  \quad \textrm{for $\alpha \in \hbox{\rm Mat}_{d_1,d_2}([0,1])$},
	\]
	satisfy
	\[
	|\Delta_n| (\log |\Delta_n|)^{d_1-1} \cdot \mathfrak{G}_{\Delta_n} \xLongrightarrow[\hbox{\tiny\rm Leb}]{} \Wei\big(2^{-(d_1-1)}\zeta(d)/c_{d_1},1\big). 
	\]
\end{theorem}

\begin{remark}
Just as before,  it can be shown (see Proposition \ref{p:log2} in the appendix) that Theorem \ref{Th_Gal} implies
Gallagher's logarithm law (Theorem \ref{th:l4}). 
\end{remark}

\subsection{Structure of the paper}

In the next section we discuss the theorems above in a very general context.  Theorem \ref{Thm_mainmu} encompasses Theorem \ref{Th_M},  Theorem \ref{Th_PLF} and Theorem \ref{Th_Pol} (see the discussion after Remark \ref{Rmk_Thmmu} for more details),  while Theorem \ref{Thm_mainmu} generalizes both Theorem \ref{Th_Dirichlet} and Theorem \ref{Th_Gal}.

\section{General framework}

Let $X_d$ 
denote the space of unimodular lattices in $\bR^d$.  Given a pair $(\eta,C)$,
where $\eta$ is a non-negative continuous function on $\bR^d$ and $C$ is a convex subset of $\bR^d$,  we define the function $\widetilde{\eta} : X_d \ra [0,\infty)$ by
\begin{equation}
\label{tildeeta}
\widetilde{\eta}(\Lambda) := \inf\big\{ \eta(\lambda) \,  : \,  \lambda \in \Lambda \cap C,  \enskip \lambda \neq 0 \big\},  \quad \textrm{for $\Lambda \in X_d$}.
\end{equation}
Motivated by the problems introduced in the previous section,
we want to investigate the asymptotic behaviour of \emph{minima} of $\widetilde{\eta}$ along \emph{sparse} samples taken from group orbits in $X_d$.  More precisely,  let $H$ be a closed subgroup of $\SL_d(\bR)$.  Given a finite subset $F \subset H$,  we define
\begin{equation}
\label{eta_Fn}
\widetilde{\eta}_F(\Lambda) := \min_{h \in F} \, \widetilde{\eta}(h.\Lambda),  \quad \textrm{for $\Lambda \in X_d$}, 
\end{equation}
and consider the following problem

\begin{problem}
Let $\theta$ be a Borel probability measure on $X_d$ and let $(F_n)$ be a sequence of finite subsets of $H$ such that $|F_n| \ra \infty$.  
Can we find a sequence $(\delta_n)$ of positive real numbers such that $(\widetilde{\eta}_{F_n}/\delta_n)$ has a non-trivial $\theta$-distributional limit?
\end{problem}

Let us make a few remarks before we state our main theorems concerning this question (see Theorem \ref{Thm_mainmu} and Theorem \ref{Thm_mainnu} below).  Note that
for all $t \geq 0$,
\begin{equation}
\label{Note1}
\big\{ \widetilde{\eta}_{F_n} < t \big\} = \bigcup_{h \in F_n} h^{-1}.\big\{ \Lambda \in X_d \, : \, \Lambda \cap C(t) \neq \{0\} \big\},  \vspace{-0.2cm}
\end{equation}
where 
$$
C(t) := \big\{ x \in C \,  : \,  \eta(x) < t \big\}.
$$
Hence,  to study the $\theta$-distributional asymptotics of the sequence $(\eta_{F_n}/\delta_n)$,  we need to understand
\[
\theta\Big( \big\{ \Lambda \in X_d \,  : \,  N_n( \Lambda ; \delta_n t ) > 0 \big\}\Big),  \quad \textrm{as $n \ra \infty$},
\]
where
\begin{equation}
\label{Def_Nn}
N_n(\Lambda ; t) := \# \big\{ h \in F_n \,  : \,  h.\Lambda \cap C(t) \neq \{0\} \big\}.
\end{equation}
In what follows,  we want to isolate conditions on $(\eta,C)$ and $(F_n)$
which ensure that we can find a positive sequence $(\delta_n)$ such that,  for some constants $a,  m_o > 0$,  
\begin{equation}
\label{PoissonTheta}
N_n( \,  \cdot \,  ; \delta_n t)  \xLongrightarrow[\theta]{} \Poi\big(m_o t^a\big),  \quad \textrm{for all $t \geq 0$}, \vspace{-0.2cm}
\end{equation}
where $\Poi(m)$ denotes the standard Poisson distribution with mean $m$,  and the arrow denotes convergence in distribution (with respect to $\theta$).  
We recall that \eqref{PoissonTheta} just means that for all $t \geq 0$ and for all non-negative integers $k$,  
\[
\lim_{n \ra \infty} \theta\Big(\{ \Lambda \in X_d \, :\,  N_n(\Lambda ; \delta_n t) = k \}\Big) 
= \frac{(m_o t^a)^k}{k!} e^{-m_o t^a}.
\]
Let us suppose that  \eqref{PoissonTheta} holds.  Then,  by summing over all $k > 0$,  we see that
\vspace{0.1cm}
\[
\theta\Big(\big\{ \widetilde{\eta}_{F_n}/\delta_n \leq t \big\}\Big) = \theta\Big( \{ \Lambda \in X_d \,  : \,  N_n( \Lambda ; \delta_n t ) > 0 \}\Big) \ra 1-e^{-m_o t^a},  
\]
\vspace{0.1cm}
for all $t \geq 0$.  In other words,  if \eqref{PoissonTheta} holds,  then
\begin{equation}
\label{WeibullTheta}
\widetilde{\eta}_{F_n}/\delta_n \xLongrightarrow[\theta]{} \Wei\big(m_o^{-1/a},a\big),   \quad \textrm{as $n \ra \infty$}.
\vspace{-0.1cm}
\end{equation}
This shows that the Weibull asymptotics of the form \eqref{WeibullTheta} can be deduced from the Poisson approximation of the form \eqref{PoissonTheta}.

\subsection{Volume regularity}

In order to establish Poisson approximation of the form \eqref{PoissonTheta} (and consequently also Weibull asymptotics of the form \eqref{WeibullTheta}),  we need to assume that the growth of the volumes of the sets $C(t)$ is stable under small perturbations.  The following definition captures a convenient form of stability which is sufficient for our purposes. 

\begin{definition}
Let $a,c > 0$ and $b \geq 0$.  The pair $(\eta,C)$ is \emph{$(a,b,c)$-regular} if 
\vspace{0.1cm}
\begin{itemize}
\item[(i)] $\eta$ is locally Lipschitz,  and $\gamma$-homogeneous for some $\gamma \geq 0$. \vspace{0.2cm}
\item[(ii)] $C(t)$ is a bounded subset of $\bR^d$ for every $t \geq 0$ and there exists a semi-norm $\tau$ such that $C = \{ \tau \leq 1\}$.  \vspace{0.2cm}
\item[(iii)] The limit
\[
c =  \lim_{t \ra 0^{+}} \frac{\Vol_d(C(t))}{t^a(-\log t)^b}
\]
exists.
\end{itemize}
\end{definition}

\vspace{0.2cm}

\noindent We give several examples of $(a,b,c)$-regular pairs $(\eta,C)$.  

\vspace{0.1cm}

\begin{example}[Norm balls]
\label{Ex1}
{\rm 
Let $\|\cdot\|$ be a norm on $\bR^d$,  and let 
\[
\eta = \tau = \|\cdot\| \qand C = \{ \tau \leq 1 \}.
\] 
Then $\eta$ is clearly Lipschitz and $1$-homogeneous,  and (iii) holds with 
\[
a = d, \enskip b = 0 \qand c = \Vol_d\big(\{\|\cdot\| \leq 1 \}\big).
\] 
}
\end{example}

\vspace{0.1cm}

\begin{example}[Products of linear forms]
\label{Ex2}
{\rm 
Let 
\[
\eta(x) = |x_1\cdots x_d|,  \qand \tau(x) = {\max}_k\, |x_k|,  \quad \textrm{for $x \in \bR^d$},
\]
and $C = \{ \tau \leq 1\}$.  Then $\eta$ is locally Lipschitz and $d$-homogeneous, and $\tau$ is a norm. To verify (iii), one can use the coordinate system and the integration formula from \cite[Lemma~4.1]{BG0}. Using the coordinate system from \cite{BG0},
the set 
$$
\{x\in \bR^d:\, 0<x_1,\ldots, x_d\le 1,\, x_1\cdots x_d\le t \}
$$
corresponds to
$$
\{(u,s)\in\bR^{d-1}\times \bR:\,\, u_1,\ldots,u_{d-1}\le 0,\,  u_1+\cdots+u_{d-1}\ge s, \, s\le \log t \},
$$
so that
$$
\Vol_d (C(t))=2^d c_d \int_{-\infty}^{\log t} e^s s^{d-1}\,ds,
$$
where $c_d$ is the volume of the simplex $\{ u \in [0,1]^{d-1} \,  : \,  u_1 + \ldots + u_{d-1} \leq 1 \big\}$. 
Hence, (iii) holds with 
\[
a = 1 \qand b = d-1 \qand c = 2^d c_d.
\]
}
\end{example}

\vspace{0.1cm}

\begin{example}[Polynomial maps]
	\label{Ex25}
{\rm 
Let $F$ be a nontrivial homogeneous polynomial map on $\bR^d$. Let
\[
\eta(x) = |F(x)|,  \qand \tau(x) = {\max}_k\, |x_k|,  \quad \textrm{for $x \in \bR^d$},
\]
and $C = \{ \tau \leq 1\}$.  Then $\eta$ is locally Lipschitz and $d$-homogeneous, and $\tau$ is a norm.
It remains to verify that the volume of the sets 
$$
C(t)=\{x\in\bR^d:\,\, |F(x)|\le t,\, |x_1|,\ldots, |x_d|\le  1 \}
$$
has the stated asymptotics as $t\to 0^+$.
This can be deduce using the resolution of singularities that there exist $c>0$,
$a\in \bQ^+$, $b\in\bN_0$ such that
$$
\Vol_d(C(t))\sim c\, t^a(-\log t)^b\quad \hbox{as $t\to 0^+$.}
$$
We refer, for instance, to \cite[Th.~1]{Gr} for a self-contained treatment.
}
\end{example}

\vspace{0.1cm}

\begin{example}[Normed strips]
\label{Ex3}
{\rm 
Let $\bd = (d_1,d_2)$ with $d = d_1 + d_2$.  Suppose that $\|\cdot\|_1$
and $\|\cdot\|_2$ are norms on $\bR^{d_1}$ and $\bR^{d_2}$ respectively.  Define
\[
\eta(x) = \|x^{(1)}\|_1 \qand \tau(x) = \|x^{(2)}\|_2,  \quad \textrm{for 
$x = (x^{(1)},x^{(2)}) \in \bR^{d_1} \times \bR^{d_2}$},
\]
and $C = \{ \tau \leq 1\}$.  Then $\eta$ is $1$-homogeneous,  $\tau$ is a semi-norm on 
$\bR^d$ and (iii) holds with
\[
a = d_1\qand b = 0 \qand c = \Vol_{d_1}\big(\{\|\cdot\|_1 \leq 1 \}\big) \cdot \Vol_{d_2}\big(\{\|\cdot\|_2 \leq 1 \}\big).
\]
}
\end{example}

\vspace{0.2cm}

\subsection{Poisson asymptotics with respect to $\mu_d$}
\label{subsec:Poissomud}

For $T\in \mathscr{T}_d$, we set
$$
h(T) := \Diag\big(T_1,\ldots,T_{d} \big) \in \SL_d(\bR).
$$
Then 
$$
H_d: = \big\{ h(T) \,  : \,  T \in \mathscr{T}_d \big\}
$$
is a closed subgroup of $\SL_d(\bR)$.
We will be interested in analyzing the minima
$\widetilde{\eta}_F(\Lambda) = \min_{h \in F} \, \widetilde{\eta}(h.\Lambda)$
as above, where $F$ runs over sparse finite subset of $H_d$
and $\Lambda$ is generic with respect to $\mu_d$.
To formulate this sparseness condition precisely, we define the distance 
\[
\dist_d(h(T),h(T')): = \max\big\{ \big|\log T_k - \log T_k' \big| \,  : \,  k =1,\ldots,d \big\},\quad T,T'\in \mathscr{T}_d.
\]
If $F$ is a finite subset of $H_d$,  we define the \emph{spread} $v_d(F)$ by
\begin{equation}
\label{def_vd}
v_d(F) := \min \big\{ \dist_d(h,h') \,  : \,  h,h' \in F,  \enskip h \neq h' \big\}.
\end{equation}
With this notation, we state our first main theorem:

\begin{theorem}
\label{Thm_mainmu}
Let $d \geq 3$ and let $(F_n)$ be a sequence of finite subsets of $H_d$ such that
\[
|F_n| \ra \infty \qand \frac{v_d(F_n)}{\log |F_n|} \ra \infty,  \quad \textrm{as $n \ra \infty$}.
\]
Suppose that the pair $(\eta,C)$ is $(a,b,c)$-regular and set
$\delta_n := |F_n|^{-1/a} \cdot (\log |F_n|)^{-b/a}$.  Define the
sequences $(\widetilde{\eta}_{F_n})$ and $(N_n)$ as in \eqref{eta_Fn} and
\eqref{Def_Nn} respectively.  Then,  
\[
N_n(\,  \cdot \,  ; \delta_n u) \xLongrightarrow[\mu_d]{} \Poi\big(m_o u^a\big),  \quad \textrm{for all $u \geq 0$},
\]
where $m_o := c a^b/2\zeta(d)$.  In particular,
\[
\widetilde{\eta}_{F_n}/\delta_n \xLongrightarrow[\mu_d]{} \Wei\big((2\zeta(d))^{1/a}/(c  a^{b})^{1/a},a\big),  \quad n \ra \infty.
\]
\end{theorem}

\begin{remark}
To get a better idea of what the condition $v_d(F_n)/\log|F_n| \ra \infty$ means,  it might be instructive to consider the case when the sets $F_n$ are contained in a fixed one-parameter family in $H_d$.  More precisely,  we define
\[
T_s = (s,\ldots,s,s^{-1/(d-1)}),  \quad \textrm{for $s > 0$}.
\]
If $0 < s_{n,1} < \ldots < s_{n,n}$,  let $F_n = \{h(T_{s_{n,1}}),\ldots,h(T_{s_{n,n}})\}$.  Then $v_d(F_n)/\log|F_n| \ra \infty$ if and only if
\[
\lim_{n \ra \infty} \frac{\log\left(\min \left\{ \frac{s_{n,k+1}}{s_{n,k}} \,  : \,  k=1,\ldots,n-1 \right\}\right)}{\log |F_n|} = \infty.
\]
In other words,  $v_d(F_n)/\log|F_n| \ra \infty$ if the lacunarity constants of the sets 
$\{s_{n,1},\ldots,s_{n,n}\}$ grow faster than $\log n$.  
\end{remark}

\begin{remark}
\label{Rmk_Thmmu}
We do not know if Theorem \ref{Thm_mainmu} holds when $d = 2$.  Our proof relies on precise small volume asymptotics in Rogers' formula (see Lemma \ref{LemmaB}),  which
we currently only know how to establish for $d \geq 3$ (see Remark \ref{Rmk_SmallVolume}). 
\end{remark}

Now we show that Theorems \ref{Th_M}--\ref{Th_Pol} are immediate consequences of  
Theorem \ref{Thm_mainmu}.

\begin{proof}[Proof of Theorem \ref{Th_M}]
Let 
$$
\eta = \tau = \|\cdot\|_\infty, \quad C=\{\tau\le 1\}=[-1,1]^d.
$$
Using the notation of Theorem \ref{Th_M}, we obtain
\begin{align*}
\widetilde \eta\big(h(T)^{-1}\Lambda\big)&=\inf\big\{ \|\rho\|_\infty:\, \rho\in h(T)^{-1}\big( \Lambda\backslash \{0\}\big) \cap C  \big\} \\
&=\inf \big\{ \|h(T)^{-1}\lambda\|_\infty:\, \lambda\in \Lambda\backslash \{0\} \cap h(T)C   \big\} \\
&=\min \big\{ {\max}_i \,T_i^{-1}|\lambda_i|\, :\,\, \lambda\in \Lambda\backslash \{0\} \cap B_T  \big\} \\
&= \min \big\{ c>0 \, :\,\, \exists \lambda\in \Lambda\backslash \{0\} \cap c B_T \big\}=c_{B_T}(\Lambda).
\end{align*}
Given a finite set $\Delta_n$, we take $F_n=\{h(T)^{-1}:\, T\in \Delta_n \}$.
Then 
$$
\mathfrak{C}_{\Delta_n}(\Lambda)=\widetilde \eta_{F_n}(\Lambda),
$$
so that we can apply Theorem \ref{Thm_mainmu} to study the minima $\mathfrak{C}_{\Delta_n}(\Lambda)$ in the Minkowski Theorem.
Finally, we note that $(\|\cdot\|_\infty,C)$ is $(d,0,2^d)$-regular (see  Example \ref{Ex1}).
\end{proof}

\begin{proof}[Proof of Theorem \ref{Th_PLF}]
Let 
$$
\eta(x) = |x_1\cdots x_d|,\quad  \tau(x) =\|x\|_\infty,\quad C = \{ \tau \leq 1\}=[-1,1]^d.
$$
Using the notation of Theorem \ref{Th_PLF}, we obtain
\begin{align*}
\widetilde \eta\big(h(T)^{-1}\Lambda\big)&=\inf\big\{ \Pi(\rho):\, \rho\in h(T)\big( \Lambda\backslash \{0\}\big) \cap C  \big\} \\
&=\inf \big\{ \Pi(\lambda):\, \lambda\in \Lambda \backslash \{0\} \cap h(T)C   \big\}\\
&= \inf \big\{ \Pi(\lambda):\, \lambda\in \Lambda \backslash \{0\} \cap B_T   \big\}=m_{B_T}(\Lambda).
\end{align*}
Hence, the minima $\mathfrak{M}_{\Delta_n}(F,\Lambda)$
can be studied with the help of Theorem \ref{Thm_mainmu}.
Finally, we note that the pair $(\eta,C)$ is $(1,d-1,2^d c_d)$-regular (see Example \ref{Ex2}).
\end{proof}	

\begin{proof}[Proof of Theorem \ref{Th_Pol}]
Let 
$$
\eta(x) = |F(x)|,\quad  \tau(x) =\|x\|_\infty,\quad C = \{ \tau \leq 1\}=[-1,1]^d.
$$
Using the notation of Theorem \ref{Th_Pol}, we obtain
\begin{align*}
\widetilde \eta(h(T)\Lambda)&=\inf\big\{ |F(\rho)|:\, \rho\in h(T)\big( \Lambda\backslash \{0\}\big) \cap C  \big\} \\
&=\inf \big\{ |F(\rho)|:\, \lambda\in \Lambda_T\backslash \{0\} \cap [-1,1]^d   \big\}=m_{T}(F,\Lambda).
\end{align*}
Hence, the minima $\mathfrak{M}_{\Delta_n}(\Lambda)$
can be studied with the help of Theorem \ref{Thm_mainmu}.
Finally, we note that the pair $(\eta,C)$ is $(a,b,c)$-regular (see Example \ref{Ex25}).
\end{proof}

\subsection{Poisson approximation with respect to the horospherical measure}
\label{subsec:Poissonud}

Our second main result involves a measure supported on the sets of lattices arising in Diophantine approximation. Let $d\ge 3$ and $\bd = (d_1,d_2)$
such that $d_1,d_2\ge 1$ and $d = d_1 + d_2$.  
Given a matrix $\alpha \in \Mat_{d_1,d_2}([0,1])$,  we define the lattice $\Lambda_\alpha \in X_d$ by
\begin{equation}
\label{def_LambdaAlpha}
\Lambda_\alpha = \big\{ (p_1 + (\alpha q)_1,\ldots,p_{d_1} + (\alpha q)_{d_1},q) \,  : \, \,
p \in \bZ^{d_1},  \enskip q \in \bZ^{d_2} \big\},
\end{equation}
and the probability measure $\nu_{\bd}$ on $X_d$ by
\[
\int_{X_d} \varphi \,   d\nu_{\bd} = \int_{\Mat_{d_1,d_2}([0,1])} \varphi(\Lambda_\alpha) \,  d\alpha,  \quad \textrm{for $\varphi \in C_b(X_d)$},
\]
where $d\alpha$ denotes the Lebesgue probability measure on 
$\Mat_{d_1,d_2}([0,1]) \cong [0,1]^{d_1 d_2}$.

Let
$$
H^{+}_{\bd} := \big\{ h(T) \,  : \,  T \in \mathscr{T}_{\bd}^{+} \big\}.
$$
For $h(T)\in H^{+}_{\bd}$, we set 
$\lfloor h(T) \rfloor_{\bd} := \lfloor T\rfloor_{\bd}$ (see \eqref{eq:T}).
If $F$ is a finite subset of $H_{\bd}^{+}$,  we define the \emph{spread} $v_{\bd}^{+}(F)$ by
\begin{equation}
\label{def_vdp}
v^{+}_{\bd}(F) := \min \big\{ \min\big(\lfloor h \rfloor,   \dist_d(h,h')\big) \,  : \,  h\ne h' \in F \big\}.
\end{equation}

Our second main result is about distribution of the minima
$\widetilde{\eta}_F(\Lambda) = \min_{h \in F} \, \widetilde{\eta}(h.\Lambda)$,
where $F$ runs over sparse finite subset of $H^+_{\bd}$,
and the lattice $\Lambda$ is generic with respect to the measure $\nu_d$.

\begin{theorem}
\label{Thm_mainnu}
Let $d \geq 3$,  and let $(F_n)$ be a sequence of finite subsets of $H^{+}_{\bd}$ such that
\[
|F_n| \ra \infty \qand \frac{v^{+}_{\bd}(F_n)}{\log |F_n|} \ra \infty,  \quad \textrm{as $n \ra \infty$}.
\]
Suppose that the pair $(\eta,C)$ is $(a,b,c)$-regular and set
$\delta_n := |F_n|^{-1/a} \cdot (\log |F_n|)^{-b/a}$.  Then,  
\[
N_n(\,  \cdot \,  ; \delta_n u) \xLongrightarrow[\nu_{\bd}]{} \Poi\big(m_o u^a\big),  \quad \textrm{for all $u \geq 0$},
\]
where $m_o := c a^b/2\zeta(d)$.  In particular,
\[
\widetilde{\eta}_{F_n}/\delta_n \xLongrightarrow[\nu_{\bd}]{} \Wei\big((2\zeta(d))^{1/a}/(c  a^{b})^{1/a},a\big),  \quad n \ra \infty.
\]
\end{theorem}

\begin{remark}
We have the same issues when $d = d_1 + d_2 = 2$ as in Remark \ref{Rmk_Thmmu}.
\end{remark}

We shall show that Theorems \ref{Th_Dirichlet} and \ref{Th_Gal} are direct consequences of Theorem \ref{Thm_mainnu}.

\begin{proof}[Proof of Theorem \ref{Th_Dirichlet}]
Let $\pi_1:\bR^{d}\to \bR^{d_1}$ and $\pi_2:\bR^{d}\to \bR^{d_2}$
denote the projections on the first $d_1$ coordinates and the last $d_2$ coordinates respectively. Let 
$$
\eta(x) = \|\pi_1(x)\|_\infty,\quad \tau(x) = \|\pi_2(x)\|_\infty,\quad
C = \{\tau\le 1\}.
$$
We note that the pair $(\eta,C)$ is $(d_1,0,2^d)$-regular (see  Example \ref{Ex3}). Then for $T\in \mathscr{T}_{\bd}^{+}$,
$$
k_T(\alpha)=\widetilde \eta\big(h(T)\Lambda_{\alpha}\big),
$$
so that Theorem \ref{Th_Dirichlet} is a consequence of Theorem \ref{Thm_mainnu}.
\end{proof}

\begin{proof}[Proof of Theorem \ref{Th_Gal}]
Let 
$$
\eta(x) = \Pi(\pi_1(x)),\quad \tau(x) = \|x\|_\infty,\quad
C = \{\tau\le 1\},
$$
where $\Pi(y)=|y_1\cdots y_{d_1}|$ for $y\in\bR^{d_1}$.
The pair $(\eta,C)$ is $(1,d_1-1,2^d c_{d_1})$-regular,
where $c_{d_1}$ is the volume of the simplex $\{ u \in [0,1]^{d_1-1} \,  : \,  u_1 + \ldots + u_{d_1-1} \leq 1 \big\}$.
Indeed, this is a combination of Examples \ref{Ex25} and \ref{Ex3}. 
Then for $T\in \mathscr{T}_{\bd}^{+}$,
$$
g_T(\alpha)=\widetilde \eta\big(h(T)\Lambda_{\alpha}\big),
$$
and Theorem \ref{Th_Gal} is a consequence of Theorem \ref{Thm_mainnu}.
\end{proof}

\subsection{Connections to earlier works}

In Section \ref{Sec:ShrinkingTargets} and Section \ref{Sec:ProofThmBn},  we explain how the results stated above can be recast in the language of dynamical systems and shrinking targets (possibly with respect to a non-invariant measure).  There is a vast literature on Poisson asymptotics (and thus on extremal laws) for shrinking targets,  and we will not attempt to make a comprehensive summary.  The investigations in this paper could be said to originate from some surprising observations concerning Poisson approximation in the gap distribution of Laplacian eigenvalues by Sinai \cite{Sinai1, Sinai2} (see also the subsequent work by Hirata \cite{Hirata} for a more dynamical framework,  closer to the spirit of this paper).  In the setting of homogeneous spaces,
the distribution of extreme values of excursions for partially hyperbolic one-parameter subgroups was studied by Pollicott \cite{pol} and Kirsebom \cite{kir}.

Recently,  Dologopyat,  Fayad and Liu \cite{DFL} have conceptualized a lot of the early developments by emphasizing the role of quantitative multiple mixing ($\cW$-mixing in our paper).  Their setup and arguments in \cite[Section 3]{DFL} are close to the ones presented in Section \ref{Sec:ShrinkingTargets} below.  In particular,  their notion of \emph{simple admissible targets} \cite[Definition 3.2]{DFL} captures a similar form of volume growth stability under perturbations as our notion of 
$(a,b,c)$-regularity,  and their \cite[Proposition 3.9]{DFL} is close to our Theorem \ref{Thm_An} (specialized to $\bR$-flows).  

\subsection{An overview of the paper}

In Section \ref{Sec:VolumeRegularity} we state a common generalization (Theorem \ref{Thm_Bn}) of Theorem \ref{Thm_mainmu} and Theorem \ref{Thm_mainnu},  and
deduce these theorems from it,  highlighting the role of volume regularity.  In Section \ref{Sec:ShrinkingTargets} we discuss Poisson asymptotics for general shrinking targets.  We introduce the notion of $\cW$-equidistribution and show that $\cW$-equidistribution ensures that we always have Poisson asymptotics under some weak assumptions on the shrinking targets (Theorem \ref{Thm_An}).  In Section \ref{Sec:ProofThmBn},  we deduce Theorem \ref{Thm_Bn} from Theorem \ref{Thm_An}.  To do this,  we need small volume asymptotics for hitting sets in the space of lattices.  Such asymptotics are proved in Section \ref{Sec:SiegelTransforms} (Lemma \ref{LemmaB}),  using classical formulas of Siegel and Rogers.

\section{The role of volume regularity}
\label{Sec:VolumeRegularity}
In this section we extend Theorem \ref{Thm_mainmu} and Theorem \ref{Thm_mainnu} to a larger class of shrinking targets (see Theorem \ref{Thm_Bn} below).  In Subsection
\ref{subsec:whyBn} we explain how $(a,b,c)$-regularity of the pair $(\eta,C)$ implies that the sequence $(C(u\delta_n))$,  which appear in both Theorem  \ref{Thm_mainmu}
and Theorem \ref{Thm_mainnu},  fit this framework,  and thereby proving these theorems (assuming Theorem \ref{Thm_Bn}).

\subsection{Definitions and notation}

We denote by $\|\cdot\|_\infty$ the $\ell^\infty$-norm on $\bR^d$,  and define
\[
D_R := \big\{ x \in \bR^d \,  : \,  \|x\|_\infty \leq R \big\},  \quad \textrm{for $R \geq 0$}. 
\]
If $t \in \bR$ and $B \subset \bR^d$,  we write
\[
t \cdot B = \{ tx \,  : \,  x \in B \big\}.  
\]
We say that a subset $B \subset \bR^d$ is \emph{balanced} if $tB \subset B$ for all $|t| \leq 1$.   Given $\eps > 0$,  let
\begin{equation}
\label{Def_Weps}
W_\eps := \big\{ g \in \SL_d(\bR) \,  : \,  \| \, g - \id \|_{\textrm{op}} < \eps \big\},
\end{equation}
where $\|\cdot\|_{\textrm{op}}$ denotes the operator norm on $\SL_d(\bR)$ with respect to $\|\cdot\|_\infty$.  If $B$ is a subset of $\bR^d$,  we write
\[
W_\eps.B = \{ g.x \,  : \,  g \in W_\eps,  \enskip x \in B \big\}.
\]

\subsection{Poisson approximation}

Let $(F_n)$ be a sequence of finite subsets of $H_d$ such that $|F_n| \ra \infty$ and let
$(B_n)$ be a sequence of bounded and balanced sets in $\bR^d$.  We define
\begin{equation}
\label{Nn_Bn}
N_n(\Lambda) := \#\big\{ h \in F_n \,  : \,  h.\Lambda \cap B_n \neq \{0\} \big\},  
\quad \textrm{for $\Lambda \in X_d$}. 
\end{equation}
In the next subsection we show how one can prove Theorem \ref{Thm_mainmu} and
Theorem \ref{Thm_mainnu} from the following theorem.  

\begin{theorem}
\label{Thm_Bn}
Let $d \geq 3$ and suppose that there exist
\begin{itemize}
\item[$(i)$] a sequence $(\eps_n)$ of positive integers such that $\eps_n \gg |F_n|^{-\sigma}$ for some $\sigma > 0$.  \vspace{0.1cm}
\item[$(ii)$] two sequences $(B_n^{-})$ and $(B_n^{+})$ of bounded and balanced 
subsets of $\bR^d$ such that 
\[
W_{\eps_n}.B_n^{-} \subset B_n \qand W_{\eps_n}.B_n \subset B_n^{+}
\]
and the limits
\[
\xi := \lim_{n \ra \infty} |F_n| \cdot \Vol_d(B_n^{+}) =  \lim_{n \ra \infty} |F_n| \cdot \Vol_d(B_n^{-})
\]
exist.
\end{itemize}
Then,
\vspace{0.1cm}
\begin{itemize}
\item[\textsc{(i)}] If $v_d(F_n)/\log |F_n| \ra \infty$,  then
\[
N_n \xLongrightarrow[\mu_d]{} \Poi\big(\xi/2\zeta(d)\big).
\]
\item[\textsc{(ii)}] If $F_n \subset H_{\bd}^{+}$ for all $n$ and $v_{\bd}^{+}(F_n)/\log |F_n| \ra \infty$,  then
\[
N_n \xLongrightarrow[\nu_{\bd}]{} \Poi\big(\xi/2\zeta(d)\big).
\]
\end{itemize}
\end{theorem}

\subsection{Proofs of Theorem \ref{Thm_mainmu} and Theorem \ref{Thm_mainnu} assuming Theorem \ref{Thm_Bn}}
\label{subsec:whyBn} 

Let $a,c > 0$ and $b \geq 0$,  and suppose that $(\eta,C)$ is a $(a,b,c)$-regular pair. 
In particular, 
\[
C(t) = \big\{ x \in \bR^d \,  : \,  \eta(x) < t,  \enskip \tau(x) \leq 1 \big\},  \quad \textrm{for all $t \geq 0$},
\] 
for some semi-norm $\tau$ on $\bR^d$,  and $C(t)$ is a balanced and bounded subset of $\bR^d$ for every $t \geq 0$.  \\

Let $(F_n)$ be a sequence of finite subsets of
$H_d$ and set $\delta_n = |F_n|^{-1/a}(\log |F_n|)^{-b/a}$.  Fix $u > 0$.  We want to apply Theorem \ref{Thm_Bn} to the sets
\begin{equation}
\label{chooseBn}
B_n = C(\delta_n u) \qand B_n^{\pm} = t_n^{\pm} \cdot C(\delta_n(u \pm \beta_n)),
\end{equation}
for some suitable sequences $(t^\pm_n)$ and $(\beta_n)$ of positive real numbers.  More precisely,  we wish to choose four sequences $(t_n^{\pm}),  (\eps_n)$ and 
$(\beta_n)$ such that 
\[
\lim_{n \ra \infty} t_n^{\pm} = 1 \qand \eps_n \gg |F_n|^{-\sigma},  \quad \textrm{for some $\sigma > 0$}, 
\] 
and (for all sufficiently large $n$,  depending on $u$),  
\vspace{0.1cm}
\begin{itemize}
\item[(1)] $W_{\eps_n}.(t_n^{-} \cdot C(\delta_n(u - \beta_n))) \subset C(\delta_n u)$.  \vspace{0.2cm}
\item[(2)] $W_{\eps_n}.C(\delta_n u) \subset t_n^{+} \cdot C(\delta_n (u+\beta_n))$. \vspace{0.2cm}
\item[(3)] $\lim_n |F_n| \Vol_d\big(C(\delta_n(u \pm \beta_n)\big) = c \cdot u^a a^b$.
\end{itemize}
\vspace{0.2cm}
If we can do this,  then Theorem \ref{Thm_Bn},  applied with $\xi = c \cdot u^a a^b$,  implies:
\vspace{0.1cm}
\begin{itemize}
\item (Theorem \ref{Thm_mainmu}) If $v_d(F_n)/\log |F_n| \ra \infty$,  then
\[
N_n \xLongrightarrow[\mu_d]{} \Poi\big(c u^a a^b/2\zeta(d)\big),  \quad \textrm{for all $u > 0$}.
\]
\item (Theorem \ref{Thm_mainnu}) If $F_n \subset H_{\bd}^{+}$ for all $n$ and $v^{+}_{\bd}(F_n)/\log |F_n| \ra \infty$,  then
\[
N_n \xLongrightarrow[\nu_{\bd}]{} \Poi\big(c u^a a^b/2\zeta(d)\big),  \quad \textrm{for all $u > 0$}.
\]
\end{itemize}
Let us now address the choices of the four sequences $(t_n^{\pm}),  (\eps_n)$ and 
$(\beta_n)$.  The following lemma will be useful.
\begin{lemma}
\label{LemmaM}
There exists a constant $M$ such that
\[
W_\eps.C(t) \subseteq (1 + M\eps) \cdot C(t+M\eps),  
\]
for all $\eps, t \in [0,1]$.
\end{lemma}

\begin{proof}
Fix $\eps,  t \in [0,1]$.  Note that
\begin{align*}
W_\eps.C(t) 
&\subseteq 
\big\{ x \in \bR^d \,  : \,  \exists \,  y \in C(t) \enskip \textrm{such that} \enskip \|x-y\|_\infty < \eps \|y\|_\infty \big\} \\[0.1cm]
&\subseteq 
\big\{ x \in \bR^d \,  : \,  \exists \,  y \enskip \textrm{such that} \enskip \|x-y\|_\infty < \eps \|y\|_\infty,   \enskip \eta(y) < t,  \enskip \tau(y) \leq 1 \big\}.
\end{align*}
Since $C(1)$ is a bounded subset of $\bR^d$,  there exists a constant $M_o$ such that
$C(1) \subset D_{M_o}$.  Hence,  for every $x \in W_\eps.C(t)$,  we can find $y \in D_{M_o}$ such that
\[
\|x\| \leq (1+\eps)\|y\|_\infty,  \enskip \eta(x) < t + |\eta(x) - \eta(y)| \qand 
\tau(x) \leq 1 + |\tau(x) - \tau(y)|.
\]
Since both $\eta$ and $\tau$ are locally Lipschitz and $x,y \in D_{2M_o}$,  we can find a constant $L$ such that
\[
|\eta(x) - \eta(y)| \leq L\|x-y\|_\infty \leq LM_o \, \eps \qand
|\tau(x) - \tau(y)| \leq L\|x-y\|_\infty \leq LM_o \,  \eps
\]
We conclude that
\[
W_\eps.C(t) \subseteq \big\{ x \in \bR^d \,  : \,  \eta(x) < t + LM_o \, \eps,  \enskip \tau(x) \leq 1 + LM_o \, \eps \big\}.
\]
Set $M := L M_o$.  Since $\eta$ is $\gamma$-homogeneous for some $\gamma \geq 0$ and $\tau$ is $1$-homogeneous,  we thus see that
\[
W_\eps. C(t) \subset (1+M\eps) \cdot C\Big( \frac{t+M\eps}{(1+M\eps)^{\gamma}}\Big) 
\subset (1 + M\eps) \cdot C(t + M\eps),
\]
which finishes the proof. 
\end{proof}

Let $(s_n)$ be a sequence of positive real numbers converging to zero (to be specified later),  and set
\[
\eps_n = \delta_n s_n \qand t_n^{-} = \frac{1}{1 + M\eps_n} \qand t_n^{+} = 1 + M\eps_n \qand \beta_n = Ms_n,
\]
where $M$ is as in the lemma above.  Note that $t_n^{\pm} \ra 1$ as $n \ra \infty$. Since $(s_n)$ and $(\delta_n)$ converge to zero,
we may without loss of generality assume that $n$ is large enough so that $\beta_n < u$ and $\delta_n u \leq 1$,  in which case we can apply Lemma \ref{LemmaM} with 
\[
t = \delta_n u \qand \eps = \eps_n = \delta_n s_n,
\]  
and readily deduce the inclusions (1) and (2) above.  Concerning the limit in (3),  we argue as follows.  Firstly, 
since $\delta_n(u \pm \beta_n) \ra 0$ and $(\eta,C)$ is $(a,b,c)$-regular,  we have
\[
c = \lim_{n \ra \infty} 
\frac{\Vol_d\big(C(\delta_n(u \pm \beta_n))\big)}{\delta_n^a(u \pm \beta_n)^a (-\log \delta_n(u \pm \beta_n))^b}.
\]
Hence,  since $\delta_n = |F_n|^{-1/a}(\log |F_n|)^{-b/a}$,  
\begin{align*}
\lim_{n \ra \infty} |F_n| \Vol_d\big(C(\delta_n(u \pm \beta_n))\big)
&= c \cdot \lim_{n \ra \infty} |F_n| \delta_n^a(u \pm \beta_n)^a (-\log \delta_n(u \pm \beta_n))^b \\[0.2cm]
&= c \cdot \lim_{n \ra \infty} \frac{(u \pm \beta_n)^a (a \log |F_n| + r^{\pm}_n(u))^b}{(\log |F_n|)^b} \\[0.2cm]
&= c \cdot  u^{a} a^b,
\end{align*}
where $r^{\pm}_n(u) := \frac{b}{a} \log \log |F_n| + \log(u \pm \beta_n)$.  This proves (3).  It remains to
choose the sequence $(s_n)$ so that $\eps_n \gg |F_n|^{-\sigma}$ for some $\sigma > 0$.  Since $\delta_n = |F_n|^{-1/a} (\log |F_n|)^{-b/a}$,  we can for instance take 
$s_n = (\log |F_n|)^{b/a}/|F_n|^{1/a}$,  in which case $\sigma = 2/a$.

\section{Poisson approximation for shrinking targets}
\label{Sec:ShrinkingTargets}

In this section we show how equidistribution (or mixing) of all orders can be used to establish Poisson asymptotics in some situations which involve shrinking targets.  We begin by proving a very general criterion for Poisson approximation (Lemma \ref{LemmaPoisson}).  Then we introduce the notion of $\cW$-equidistribution which is a simultaneous generalization of both (quantitative) equidistribution (of all orders) and (quantitative) mixing (of all orders).  Our main result in this section is Theorem \ref{Thm_An},  which provides criteria for when the hitting times to shrinking targets in a $\cW$-equidistributing setting are asymptotically Poissonian. 

\subsection{A general criterion for Poisson approximation}

Let $(Z,\nu)$ be a probability space and let $H$ be an infinite set.  Let $(F_n)$ be a sequence of finite subsets of $H$ such that $|F_n| \ra \infty$.  Suppose that for every 
$n$ and $h \in F_n$,  we are given a measurable subset $A_{n,h} \subset Z$.  Define 
\begin{equation}
\label{Def_Nn_General}
N_n(z) := \#\big\{ h \in F_n \,  : \,  z \in A_{n,h} \big\},  \quad \textrm{for $z \in Z$}. 
\end{equation}
We will prove that the sequence $(N_n)$ is $\nu$-asymptotically Poissonian with mean $m$,  provided that the sets $A_{n,h}$,  for $h \in F_n$,  are "almost independent" for large $n$.  To make this more precise,  we define
\begin{equation}
\Delta_{n,r}(m) :=  \max\Big\{ \big| \,  |F_n|^r \cdot \nu\Big( \bigcap_{h \in F'} A_{n,h} \Big) - m^r \big| \,  : \,  F' \subset F_n, \enskip |F'| = r \Big\},
\end{equation}
for positive integers $r,n$ and a real number $m \geq 0$.  If $|F_n| < r$,   we set $\Delta_{n,r}(m) = 0$ for all $m$.  Note that if the sets $A_{n,h}$ for $h \in F_n$ are
$\nu$-independent,  and if $\nu(A_{n,h}) = m$ for all $h$,  then $\Delta_{n,r}(m) = 0$.  \\

The following lemma will be useful in the next subsection. 

\begin{lemma}
\label{LemmaPoisson}
Suppose that there exists $m \geq 0$ such that
\[
\lim_{n \ra \infty} \Delta_{n,r}(m) = 0,  \quad \textrm{for all $r \geq 1$}.
\]
Then 
$$
N_n \xLongrightarrow[\nu]{} \Poi(m)\quad\hbox{as $n \ra \infty$. }
$$
\end{lemma}

\begin{proof}
By the classical method of moments (see e.g.  \cite[Theorem 30.1]{Billingsley}),  it
suffices to show that
\begin{equation}
\label{suffice}
\lim_{n \ra \infty} \int_Z \binom{N_n}{r} \,  d\nu = \frac{m^r}{r!},  \quad \textrm{for all $r \geq 1$}. 
\end{equation}
By \cite[Section 7.3,  Eq.  (3.4)]{Ross},  
\[
\int_Z \binom{N_n}{r} \,  d\nu  = \sum_{F'} \nu\Big( \bigcap_{h \in F'} A_{n,h} \Big),   \quad \textrm{for all $r \geq 1$},
\]
where the sum is taken over all finite subsets $F' \subset F_n$ with $|F'| = r$.  Hence,  
\begin{align*}
\int_Z \binom{N_n}{r} \,  d\nu  
&= 
\sum_{F'} \Big( \nu\Big( \bigcap_{h \in F'} A_{n,h} \Big) - \frac{m^r}{|F_n|^r} \Big) + 
\binom{|F_n|}{r} \frac{m^r}{|F_n|^r} \\[0.2cm]
&= \frac{1}{|F_n|^r} \cdot \sum_{F'} \Big( |F_n|^r \nu\Big( \bigcap_{h \in F'} A_{n,h} \Big) - m^r \Big) + \binom{|F_n|}{r} \frac{m^r}{|F_n|^r}.
\end{align*}
Since $|F_n| \ra \infty$ as $n \ra \infty$ and $\lim_{n \ra \infty} \Delta_{n,r}(m) = 0$,  we conclude that
\begin{align*}
\varlimsup_{n \ra \infty} \Big| \int_Z \binom{N_n}{r} \,  d\nu - \frac{m^r}{r!} \Big| 
&\leq \varlimsup_{n \ra \infty}\Big( \,  \frac{1}{|F_n|^r} \binom{|F_n|}{r} \,  \Delta_{n,r}(t)
+ \Big(\frac{1}{|F_n|^r} \binom{|F_n|}{r} -  \frac{1}{r!}\Big) \cdot m^r \,  \Big) \\[0.2cm]
&= \varlimsup_{n \ra \infty} \frac{\Delta_{n,r}(m)}{r!} = 0.
\end{align*}
Since $r$ is arbitrary,  this proves \eqref{suffice},  and we are done.
\end{proof}

\subsection{Quantitative multiple equidistribution}

The main aim of this subsection is to define the notion of $\cW$-equidistribution for actions of real Lie groups on locally compact spaces.   To be able to do this,  we first need to discuss $C^q$-norms and good pairs.  

\subsubsection{$C^q$-norms}

Let $G$ be a real Lie group with Lie algebra $\gog$.  Given $Y \in \gog$ we denote by
$D_Y$ the differential operator defined by
\[
(D_Y \rho)(g) = \frac{d}{dt} \rho(e^{tY}g) \mid_{t = 0},  \quad \varphi \in C_c^\infty(G).
\]
Let $\{Y_1,\ldots,Y_{s}\}$ be a basis for $\gog$,  where $s = \dim \gog$.  Given 
$\alpha \in \bN^{s}$,  we define the differential operator $D_\alpha$ by
\[
D_\alpha := D_{Y_1}^{\alpha_1} \cdots D_{Y_{s}}^{\alpha_{s}}.
\]
We refer to $|\alpha| := \alpha_1 + \ldots + \alpha_s$ as the \emph{order} of $D_\alpha$,  and define,  for $q \geq 1$,  
\begin{equation}
\label{def_Cq}
\|\rho\|_{C^q} := \max\{ \|D_\alpha \rho\|_\infty \,  : \,  |\alpha| \leq q \big\},   \quad \textrm{for $\rho \in C_c^\infty(G)$}.
\end{equation}

\subsubsection{Good pairs and $\cW$-equidistribution}

Let $Z$ be a locally compact and second countable metrizable space,  endowed with a jointly continuous action of $G$.  We assume that there exists a $G$-invariant Borel probability measure $\mu$ on $Z$.  Let $\cA$ be a $G$-invariant sub-algebra of 
$\bC \cdot 1 + C_c(Z)$,  and let $\cS = (\cS_q)$ be a family of semi-norms on $\cA$. 

\begin{definition}[Good pairs]
\label{Def_good}
We say that the pair $(\cA,\cS)$ is \emph{good} if for every Borel set $A \subset Z$ such that either $A$ or $A^c$ is pre-compact,  and for every $\rho \in C_c^\infty(G)$,  we have
\[
\rho * \chi_A \in \cA \qand \cS_q(\rho * \chi_A) \ll_q \|\rho\|_{C^q},
\]
for all $q \geq 1$,  where the implicit constants only depend on $q$,  and not on the set $A$. 
\end{definition}

Let $H$ be a closed subgroup of $G$.  For every $r \geq 1$,  let $w_r$ be a non-negative function on the space of $r$-element subsets of $H$,  and set $\cW = (w_r)$.

\begin{definition}[$\cW$-equidistribution and $\cW$-mixing]
\label{Def_Wequi}
Let $\nu$ be a Borel probability measure on $Z$.  We say that \emph{$\nu$ $\cW$-equidistributes to $\mu$ along $H$  with respect to the pair $(\cA,\cS)$} if for every $r \geq 1$,  there exist $\delta_r > 0$ and $q_r \geq 1$ such that for every $\varphi \in \cA$ and $F \subset H$ with $|F| = r$,  we have
\[
\Big| \int_Z \prod_{h \in F} \varphi \circ h \,  d\nu - \Big(\int_Z \varphi \,  d\mu \Big)^r \Big| 
\ll_r e^{-\delta_r w_r(F)} \,  \cS_{q_r}(\varphi)^r,
\]
where the implicit constants only depend on $r$.  In the case when $\nu = \mu$,  we
say that \emph{$\mu$ is $\cW$-mixing with respect to $(\cA,\cS)$}.
\end{definition}

In Section \ref{Sec:ProofThmBn} below,  we consider the case $Z = X_d$ and $G = \SL_d(\bR)$ and
$\mu = \mu_d$,  and provide examples of good pairs $(\cA,\cS)$ and 
triples $(H,\nu, \cW)$ such that $\nu$ equidistributes to $\mu$ along $H$ with respect to the pair $(\cA,\cS)$.

\subsection{$\cW$-equidistribution and Poisson approximation}

Our aim is now to connect $\cW$-equidistribution with Poisson approximation.  This is done in Theorem \ref{Thm_An} below.  \\

In what follows,  let $H$ be a closed subgroup of $G$ and let $(\cA,\cS)$ be a 
good pair.  Let $\nu$ be a probability measure on $Z$ and suppose that $\nu$ 
$\cW$-equidistributes to $\mu$ along $H$ with respect to $(\cA,\cS)$,  for 
some $\cW = (w_r)$.  For every $r \geq 1$,  we extend $w_r$ to a function 
$\widetilde{w}_r$ defined on all finite subsets of $H$ by 
\[
\widetilde{w}_r(F) = \min\{ w_r(F') \,  : \,  F' \subset F,  \enskip |F'| = r \big\}.
\]
If $|F| < r$,  we set $\widetilde{w}_r(F) = 0$.  \\

Let $(F_n)$ be a sequence of finite subsets of $H$ such that $|F_n| \ra \infty$,  and
let $(A_n)$ be a sequence of Borel subsets of $Z$ such that for every $n$,  either 
$A_n$ or $A_n^c$ is pre-compact.  We define
\begin{equation}
\label{Nn_An}
N_n(z) = \#\big\{ h \in F_n \,  : \,  h.z \in A_n \big\},  \quad \textrm{for $z \in Z$}.
\end{equation}
The next theorem is the main result of this section.  It provides a criterion for when
the sequence $(N_n)$ is $\nu$-asymptotically Poissonian in terms of $\cW$ and the behaviour of the sequence $\mu(A^{\pm}_n)$ of real numbers,  where $A_n^{\pm}$ are Borel sets in $Z$ which approximate the set $A_n$ from above and below. 
The spirit of the criterion is similar to \cite[Theorem 2.10 and Theorem 3.7]{DFL},  which deal with $\bR$-flows.  

\begin{theorem}
\label{Thm_An}
Suppose that
\[
\lim_{n \ra \infty} \frac{\widetilde{w}_r(F_n)}{\log |F_n|} = \infty,  \quad \textrm{for all $r \geq 1$}.
\]
and that there exist
\vspace{0.1cm}
\begin{itemize}
\item[$(i)$] a sequence $(V_n)$ of open identity neighbourhoods in $G$ and a sequence $(\rho_n)$ of non-negative smooth functions on $G$ such that
$\supp (\rho_n) \subset V_n$ and $\int_G \rho_n \,  dm_G = 1$ for all $n$,  with the property that for every $q \geq 1$,  there is $\sigma_q > 0$,  such that
\[
\|\rho_n\|_{\cC^{q}(G)} \ll_q |F_n|^{\sigma_q},  
\] 
for all $n$,  where the implicit constants only depend on $q$.   \vspace{0.1cm}
\item[$(ii)$] two sequences $(A_n^{-})$ and $(A_n^{+})$ of Borel sets in $Z$ such that 
\[
V_n.A_n^{-} \subset A_n \qand V_n.A_n \subset A_n^{+},  \quad \textrm{for all $n$},
\]
and such that 
\[
m := \lim_{n \ra \infty} |F_n| \mu(A_n^{+}) =  \lim_{n \ra \infty} |F_n|  \mu(A_n^{-}).
\]
\end{itemize}
\vspace{0.1cm}
Then,  $N_n \xLongrightarrow[\nu]{} \Poi(m)$,  as $n \ra \infty$.
\end{theorem}

\subsection{Proof of Theorem \ref{Thm_An}}

We define $\varphi_n^{\pm} = \rho_n * \chi_{A_n^{\pm}}$.  Since $\rho_n \geq 0$ and $\supp(\rho_n) \subset V_n$,  the inclusions in our assumption (ii) imply
\[
\varphi_n^{-} \leq \chi_{A_n} \leq \varphi_n^{+},
\]
and thus,  for every finite subset $F \subset H$
\begin{equation}
\label{incl}
\int_{Z} \prod_{h \in F} \varphi_n^{-} \circ h \,  d\nu
\leq
\nu\Big(\bigcap_{h \in F'} h^{-1}.A_n\Big)
\leq \int_{Z} \prod_{h \in F'} \varphi_n^{+} \circ h \,  d\nu.
\end{equation}
Furthermore,  since $\mu$ is $G$-invariant and $\int_G \rho_n \,  dm_G = 1$,  we have
\begin{equation}
\label{measphipm}
\int_Z \varphi^{+}_n \,  d\mu = \mu(A_n^{+}) \qand \int_Z \varphi_n^{-} \,  d\mu = \mu(A_n^{-}). 
\end{equation}
Since $(\cA,\cS)$ is a good pair,  and since either $A_n$ of $A_n^c$ is pre-compact,  we have $\varphi_n^{\pm} \in \cA$,  and
\begin{equation}
\label{Sq}
\cS_q(\varphi_n^{\pm}) \ll_q \|\rho_n\|_{C^{q}(G)} \ll_q |F_n|^{\sigma_{q}},  \quad \textrm{for all $q \geq 1$},
\end{equation}
where we in the last inequality used our assumption (i).  We stress that the implicit constants are independent of $n$.  \\

By assumption,  $\nu$ $\cW$-equidistributes to $\mu$ along $H$ with respect to $(\cA,\cS)$,  so it follows from  \eqref{measphipm} and \eqref{Sq} that
\begin{equation}
\label{boundAnpm}
\Big| 
\int_X \prod_{h \in F'} \varphi_n^{\pm} \circ h \,  d\nu - \mu(A_n^{\pm})^r 
\Big| \ll_r |F_n|^{r \sigma_{q_r}} \cdot e^{-\delta_r w_r(F')}
\end{equation}
for every subset $F' \subset F_n$ with $r$ elements.  We stress that the implicit constants only depend on $r$.   From \eqref{incl} and \eqref{boundAnpm} we now
conclude that
\begin{align}
|F_n|^r \nu\Big(\bigcap_{h \in F'} h^{-1}.A_n) - m^r
&\leq 
|F_n|^r \Big( \int_Z \prod_{h \in F'} \varphi_n^{+} \circ h - \mu(A_n^{+})^r \Big)
+
\big(|F_n| \mu(A_n^{+})\big)^r  - m^r \nonumber \\[0.2cm]
&\ll_r
\big(|F_n| \mu(A_n^{+})\big)^r  - m^r + |F_n|^{r(1+\sigma_{q_r})} \cdot e^{-\delta_r w_r(F')} 
\nonumber \\[0.3cm]
&\ll_r
\big(|F_n| \mu(A_n^{+})\big)^r  - m^r + |F_n|^{r(1+\sigma_{q_r})}\cdot e^{-\delta_r \widetilde{w}_r(F_n)}, \label{up1}
\end{align}
for every subset $F'$ of $F_n$ with $r$ elements.  Similarly, 
\begin{equation}
\label{down1}
|F_n|^r \nu\Big(\bigcap_{h \in F'} h^{-1}.A_n) - m^r \gg_r 
\big(|F_n| \mu(A_n^{-})\big)^r  - m^r - |F_n|^{r(1+\sigma_{q_r})} \cdot e^{-\delta_r \widetilde{w}_r(F_n)},
\end{equation}
for every subset $F'$ of $F_n$ with $r$ elements.  \\

Since we assume that 
\[
\lim_{n \ra \infty} \frac{\widetilde{w}_r(F_n)}{\log |F_n|} = \infty  \qand 
m = \lim_{n \ra \infty} |F_n| \mu(A_n^{+}) = \lim_{n \ra \infty} |F_n| \mu(A_n^{-}),  
\]
we see from \eqref{up1} and \eqref{down1} that
\[
\lim_{n \ra \infty} \max\Big\{ \Big| |F_n|^r \nu\Big(\bigcap_{h \in F'} h^{-1}.A_n\Big) - m^r \Big| \,  : \,  F' \subset F_j,  \enskip |F'| = r \Big\} = 0,
\]
for every $r \geq 1$.   By Lemma \ref{LemmaPoisson},  this shows that $N_n \xLongrightarrow[\nu]{} \Poi(m)$,  as $n \ra \infty$.

\section{Siegel transforms}
\label{Sec:SiegelTransforms}

We recall that $X_d$ denotes the space of unimodular lattices in $\bR^d$,  equipped with the Chabauty topology.  The standard $\SL_d(\bR)$-action on $X_d$ is jointly
continuous and there is a unique $\SL_d(\bR)$-invariant Borel probability measure 
on $X_d$,  which we denote by $\mu_d$.

\subsection{Hitting sets}

Given a Borel set $B \subset \bR^d$,  we define the \emph{hitting set} 
$\Omega_B \subset X_d$ by
\begin{equation}
\label{def_OmegaB}
\Omega_B = \big\{ \Lambda \in X_d \,  : \,  \Lambda \cap B \neq \{0\} \big\}. 
\end{equation}
We note that $\Omega_B$ is a Borel set in $X_d$.  By Mahler's Compactness Criterion \cite[Chapter V,  Theorem IV]{Cassels},  the complement $\Omega_B^c$ is compact in $X_d$ if $B$ contains a neighborhood of $0$ in $\bR^d$.

\subsection{Siegel transforms}

Given $\Lambda \in X_d$,  we denote by $\Lambda_{\textrm{pr}}$ the set of primitive vectors in $\Lambda$.  If $f : \bR^d \ra \bC$ is a Borel function with bounded support,  we define its \emph{Siegel transform} $\widehat{f} : X_d \ra \bC$ by
\begin{equation}
\widehat{f}(\Lambda) = \hspace{-0.2cm} \sum_{\hspace{0.2cm} \lambda \in \Lambda_{\textrm{pr}}} \hspace{-0.1cm}f(\lambda),  \quad \textrm{for $\Lambda \in X_d$}.
\end{equation}
In particular,  if $B$ is a bounded Borel set in $\bR^d$,  then $\widehat{\chi}_B(\Lambda) = |B \cap \Lambda_{\textrm{pr}}|$.  \\

The following theorem summarizes important special cases of two fundamental formulas of Siegel and Rogers.  

\begin{theorem}[Siegel's and Rogers' formulas]
\label{Thm_SiegelRogers}
Let $B \subset \bR^d$ be a bounded and symmetric Borel set.  
\vspace{0.1cm}
\begin{itemize}
\item[$(i)$] \textsc{Siegel,  \cite[Section 2]{Siegel}} If $d \geq 2$,  then $\widehat{\chi}_B \in L^1(\mu_d)$ and 
\[
\int_{X_d} \widehat{\chi}_B(\Lambda) \,  d\mu_d(\Lambda) = \frac{\Vol_d(B)}{\zeta(d)},
\]
\item[$(ii)$] \textsc{Rogers,   \cite[Theorem 5]{Rogers}} If $d \geq 3$,  then $\widehat{\chi}_B \in L^2(\mu_d)$,  and
\[
\int_{X_d} \widehat{\chi}_B(\Lambda)^2 \,  d\mu_d(\Lambda) = 
2 \frac{\Vol_d(B)}{\zeta(d)}  + \Big( \frac{\Vol_d(B)}{\zeta(d)}\Big)^2.
\]
\end{itemize}
\end{theorem}

\subsection{Volume asymptotics}

We recall that a subset $B \subset \bR^d$ is \emph{balanced} if $t  B \subset B$ for all $|t| \leq 1$.  Note that every balanced set is symmetric.  

\begin{lemma}
\label{LemmaOmegaB}
Let $d \geq 2$.  For every bounded and balanced Borel set $B \subset \bR^d$,  we have
\[
\Omega_B = \{ \widehat{\chi}_{B} > 0 \} = \{ \widehat{\chi}_{B} \geq 2 \}.
\]
and
\[
\mu_d(\Omega_B) \leq \frac{\Vol_d(B)}{2 \zeta(d)}.
\]
\end{lemma}

\begin{proof}
Since $B$ is balanced,  we have $\frac{1}{k} B \subseteq B$ for every integer $k \geq 1$,  and thus
\vspace{0.1cm}
\begin{align*}
\Omega_B 
&= \Big\{\Lambda \,  : \,  \Lambda \cap B \neq \{0\} \Big\} =\Big\{ \Lambda \,  : \,  \bigcup_{k \geq 1} \Big( k.\Lambda_{\textrm{pr}} \cap B \Big)  \neq \{0\} \Big\} \\
&=
\Big\{ \Lambda \,  : \,  \bigcup_{k \geq 1} k.  \Big(\Lambda_{\textrm{pr}} \cap \frac{1}{k}B \Big)  \neq \{0\} \Big\} = 
\big\{ \Lambda \,  : \,  \Lambda_{\textrm{pr}} \cap B  \neq \{0\} \big\}.
\end{align*}
This shows that $\Omega_B = \{ \widehat{\chi}_B > 0 \}$.  Furthermore,  since $B$
is symmetric,  $|\Lambda_{\textrm{pr}} \cap B|$ must be an even integer for every $\Lambda \in X_d$,  which proves that we also have $\Omega_B = \{ \widehat{\chi}_{B} \geq 2 \}$.  Finally,  we note that by Siegel's formula (Theorem \ref{Thm_SiegelRogers} (i)),  
\[
\frac{\Vol_d(B)}{\zeta(d)} = \int_{X_d} \widehat{\chi}_B \,  d\mu_d \geq 
\int_{\{\widehat{\chi}_B \geq 2\}} \widehat{\chi}_B \,  d\mu_d \geq 2 \cdot \mu_d(\Omega_B).
\]
\end{proof}

\begin{lemma}
\label{LemmaB}
Let $d \geq 3$.  For every bounded and balanced Borel set $B \subset \bR^d$, 
we have
\[
\Big| \, \mu_d(\Omega_B) - \frac{\Vol_d(B)}{2\zeta(d)} \,  \Big| \leq \Big( \frac{\Vol_d(B)}{2\zeta(d)} \Big)^2.
\]
\end{lemma}

\begin{proof}
By Lemma \ref{LemmaOmegaB},  we have $\mu_d(\Omega_B) \leq \frac{\Vol_d(B)}{2 \zeta(d)}$,  so we only need to show that
\begin{equation}
\label{want}
\mu_d(\Omega_B) \geq \frac{\Vol_d(B)}{2\zeta(d)} - \Big(\frac{\Vol_d(B)}{2\zeta(d)} \Big)^2.
\end{equation}
We note that
\[
\int_{X_d} \widehat{\chi}_B \,  d\mu_d = 
\int_{X_d} \chi_{\{\widehat{\chi}_B > 0\}} \cdot \widehat{\chi}_B \,  d\mu_d
\leq \mu_d(\Omega_B)^{1/2} \cdot \Big( \int_{X_d} \widehat{\chi}_B^2 \,  d\mu_d \Big)^{1/2},
\]
and thus,  by Theorem \ref{Thm_SiegelRogers} (i) and (ii).
\[
\Big( \frac{\Vol_d(B)}{\zeta(d)} \Big)^2 \leq \mu_d(\Omega_B) \cdot \Big( 2 \frac{\Vol_d(B)}{\zeta(d)}  + \Big( \frac{\Vol_d(B)}{\zeta(d)}\Big)^2 \Big).
\]
From this we conclude that
\[
\mu_d(\Omega_B) \geq \frac{\Vol_d(B)}{2\zeta(d)} \cdot \frac{1}{1 +\frac{\Vol_d(B)}{2\zeta(d)}} \geq \frac{\Vol_d(B)}{2\zeta(d)} - \Big(\frac{\Vol_d(B)}{2\zeta(d)}\Big)^2,
\]
where we in the last inequality used the trivial inequality $\frac{1}{1+t} \geq 1 - t$ 
for all $t \geq 0$.  
\end{proof}

\begin{remark}
\label{Rmk_SmallVolume}
Kleinbock and Margulis have proved a version of Lemma \ref{LemmaB} for \emph{convex} sets in \cite[Proposition 7.1]{KM}.  They claim that their
lemma also holds for $d = 2$,  but unfortunately their argument relies on 
\cite[Theorem 7.3]{KM},  which cannot be applied in the case when $d=k=2$ (the relevant parameters in \cite[Proposition 7.1]{KM} for $d = 2$).  

There are some 
versions of Roger's formula for $d = 2$ by Athreya-Margulis \cite[Subsection 4.2]{AM} and Kelmer-Yu \cite[Theorem 1]{KY}.  However,  it seems to be difficult to use these
formulas to deduce small volume asymptotics as in Lemma \ref{LemmaB}.
\end{remark}

\section{Proof of Theorem \ref{Thm_Bn}}
\label{Sec:ProofThmBn}

\subsection{A good pair}

Let $\cA = \bC \cdot 1 + C^\infty_c(X_d)$.  Note that if $A$ is Borel subset of $X_d$
such that either $A$ or its complement $A^c$ is compact,  and $\rho$ is a complex-valued,  smooth and compactly supported function on $G$,  then $\rho * \chi_A \in \cA$.  We will construct a family $\cS = (\cS_q)$
of norms on the algebra $\cA$ such that the pair $(\cA,\cS)$ is good (see Definition \ref{Def_good}).  Given a non-zero element $Y$ in the Lie algebra  $\gos\gol_d(\bR)$,  we define the differential operator $D_Y$ on $\cA$ by
\[
D_Y \varphi(x) = \frac{d}{dt}\varphi(e^{tY}.x) \mid_{t=0},  \quad \textrm{for $x \in X_d$}.
\]
Let $\{Y_1,\ldots,Y_{d^2-1}\}$ be a basis for $\gos\gol_d(\bR)$.  For every 
$\alpha = (\alpha_1,\ldots,\alpha_{d^2-1}) \in \bN^{d^2-1}$,  we set
$D_\alpha := D_{1}^{\alpha_1} \cdots D_{d^2-1}^{\alpha_{d^2-1}}$,
and for $q \geq 1$ and $\varphi \in \cA$,  we define 
\[
\|\varphi\|_{C^q(X_d)} = \max\{ \|D_\alpha \varphi\|_{\infty} \,  : \,  \alpha_1 + \ldots + \alpha_{d^2 - 1} \leq q \big\}.
\]
Let us fix a proper right-invariant distance function on $\SL_d(\bR)$.  Since $\SL_d(\bR)$ acts transitively on $X_d$ with uniformly discrete stabilizers,  we can define a (non-invariant) quotient metric on $X_d$,  and write $\Lip$ for the corresponding Lipschitz norm.  We now define
\begin{equation}
S_q(\varphi) = \max(\|\varphi\|_{C^q(X_d)},\Lip(\varphi)),  \quad \textrm{for $q \geq 1$}.
\end{equation}
It is not difficult to see that if $A \subset X_d$ is a Borel set such that either $A$ or $A^c$ is pre-compact,  and $\rho \in C_c^\infty(G)$,  then 
\[
S_q(\rho * \chi_A) \ll_q \|\rho\|_{C^q},  \quad \textrm{for all $q \geq 1$},
\]
where the implict constants are independent of the Borel set $A$,  and $\|\cdot\|_{C^q}$ is defined as in \eqref{def_Cq}.  Hence $(\cA,\cS)$ is a good pair.  \\

We recall that
\[
W_\eps = \big\{ g \in \SL_d(\bR) \,  : \,  \|g - \id\|_{\textrm{op}} < \eps \big\}, \quad \textrm{for $\eps > 0$},
\]
where $\|\cdot\|_{\textrm{op}}$ is the operator norm with respect to the $\ell^\infty$-norm on $\bR^d$.  The proof of the following simple lemma is left to the reader.  

\begin{lemma}
\label{Lemma_rhoeps}
There is a family $(\widetilde{\rho}_\eps)$ of non-negative smooth functions on $\SL_d(\bR)$
such that $\supp \widetilde{\rho}_\eps \subset W_\eps$ and $\int \widetilde{\rho}_\eps \,  dm_{\SL_d(R)} = 1$
for every $\eps > 0$,  and with the property that for every $q \geq 1$,  there exists $\kappa_q > 0$ such that
\[
\|\widetilde{\rho}_\eps\|_{C^q} \ll \eps^{-\kappa_q},  \quad \textrm{for all $\eps > 0$},
\]
where the implicit constants only depend on $q$.
\end{lemma}

\subsection{$\cW$-mixing for $\mu_{d}$}

Let $H_d < \SL_d(\bR)$ be as in Subsection \ref{subsec:Poissomud}.  If $F \subset H_d$ is a finite set with $r$ elements,  we define $w_r(F) := v_d(F)$,  where $v_d$ is defined as in \eqref{def_vd},  and set $\cW = (w_r)$.  The following theorem was proved by the authors and Manfred Einsiedler \cite[Theorem 1.1]{BEG}.  

\begin{theorem}
\label{Thm_BEG}
For every $d \geq 2$,  $\mu_d$ is $\cW$-mixing with respect to $(\cA,\cS)$.
\end{theorem}

\begin{remark}
Theorem \ref{Thm_BEG} is not explicitly stated in \cite{BEG} with the norms $(\cS_q)$ above.  However,  these norms satisfy the properties (N1-N4) listed in \cite[Subsection 2.2]{BEG}; this is enough to guarantee $\cW$-mixing.
\end{remark}

\subsection{$\cW$-equidistribution of $\nu_{\bd}$ to $\mu_d$ along $H^{+}_{\bd}$}

Let $\bd = (d_1,d_2)$ such that $d = d_1 + d_2$,  and let $H^{+}_{\bd} < \SL_d(\bR)$ be as in Subsection \ref{subsec:Poissonud}.  If $F \subset H_d^{+}$ is a finite set with
$r$ elements,  we define $w_r(F) := v_{\bd}^{+}(F)$,  where $v_{\bd}^{+}$ is defined as in \eqref{def_vdp}.  The following theorem was proved by the authors 
\cite[Theorem 1.1]{BG2}.  

\begin{theorem}
\label{Thm_BG}
For every $\bd = (d_1,d_2)$ such that $d = d_1 + d_2 \geq 2$,  $\nu_{\bd}$ $\cW$-equidistributes to $\mu_d$ along $H_{\bd}^{+}$.
\end{theorem}

\begin{remark}
Theorem \ref{Thm_BG} for one-parameter sub-semigroups of $H_d^{+}$ was proved earlier by the authors in \cite[Theorem 2.2]{BG1}.
\end{remark}

\subsection{Proof of Theorem \ref{Thm_Bn}}

Let $d \geq 3$,  and let $(F_n)$ be a sequence of finite subsets of $H_d$ such that $|F_n| \ra \infty$
and \emph{either}
\vspace{0.1cm}
\begin{itemize}
\item[(1)] $\displaystyle \lim_{n \ra \infty} \frac{v_d(F_n)}{\log |F_n|} = \infty$,  or \vspace{0.1cm}
\item[(2)] $F_n \subset H_{\bd}^{+}$ and $\displaystyle \lim_{n \ra \infty} \frac{v^{+}_{\bd}(F_n)}{\log |F_n|} = \infty$.
\end{itemize}
Theorem \ref{Thm_BEG} tells us that we are in the setting of Theorem \ref{Thm_An},  and (1) implies that the first assumption of Theorem \ref{Thm_An} (about $\widetilde{w}_r$) is satisfied.  Theorem \ref{Thm_BG},  in combination with (2),  says the same thing.  
We will check that the other two conditions (i) and (ii) (with $m = \xi/2\zeta(d))$ in Theorem \ref{Thm_An} are satisfied for 
\[
V_n := W_{\eps_n} \qand \rho_n := \widetilde{\rho}_{\eps_n}
\qand
A_n := \Omega_{B_n} \qand A_n^{\pm} := \Omega_{B_n^{\pm}},
\]
where $(\eps_n)$,  $(B_n)$ and $(B_n^{\pm})$ are as in the statement of Theorem \ref{Thm_Bn}.  If these conditions hold,  then Theorem \ref{Thm_An} implies that 
\vspace{0.1cm}
\begin{itemize}
\item $N_n \xLongrightarrow[\mu_d]{} \Poi\big(\xi/2\zeta(d)\big)$ - in the case of (1).  
\item $N_n \xLongrightarrow[\nu_{\bd}]{} \Poi\big(\xi/2\zeta(d)\big)$ - in the case of (2).  
\end{itemize}

\subsubsection*{Proof of (i) in Theorem \ref{Thm_An}}
Since $\eps_n \gg |F_n|^{-\sigma}$ for some $\sigma > 0$,  we conclude from Lemma \ref{Lemma_rhoeps} that
\[
\|\rho_n\|_{C^q} \ll \eps_n^{-\kappa_q} \ll |F_n|^{\sigma \kappa_q},  \quad \textrm{for all $q \geq 1$},
\]
where the implicit constants only depend on $q$. This proves (i) with $\sigma_q := \sigma \kappa_q$. \\

\subsubsection*{Proof of (ii) in Theorem \ref{Thm_An}}

The inclusions are obvious.  To prove that
\[
\lim_{n \ra \infty} |F_n| \mu_d(A_n^{\pm}) = \frac{\xi}{2\zeta(d)},
\]
we argue as follows.  First note that
\begin{equation}
\label{FnAn}
|F_n| \cdot \mu_d(A_n^{\pm})
= 
|F_n| \cdot \Big( \mu_d(\Omega_{B_n^{\pm}}) - \frac{\Vol_d(B_n^{\pm})}{2\zeta(d)} \Big) +
|F_n| \cdot \frac{\Vol_d(B_n^{\pm})}{2\zeta(d)}.
\end{equation}
By our assumption in Theorem \ref{Thm_Bn},  the second term converges to 
$\xi/2\zeta(d)$ when $n \ra \infty$.  In particular,  $\Vol_d(B_n^{\pm}) \ra 0$.  By Lemma \ref{LemmaB},  the absolute value of the first term is bounded from above by
\[
|F_n| \cdot \Big( \frac{\Vol_d(B_{n}^{\pm})}{2\zeta(d)}\Big)^2 = 
\Big( |F_n| \cdot \frac{\Vol_d(B_{n}^{\pm})}{2\zeta(d)}\Big) \cdot \frac{\Vol_d(B_{n}^{\pm})}{2\zeta(d)},
\]
Since the first factor on the right hand side stays bounded as $n \ra \infty$ while the second factor tends to zero,  we conclude that the first term in \eqref{FnAn} tends to
zero as $n \ra \infty$.  Hence,  $\lim_n |F_n| \,  \mu(A_n^{\pm}) = \xi/2\zeta(d)$,  which finishes the proof of (ii).

\appendix\section{Weibull asymptotics implies a logarithm Law}

Let $(X,\mu)$ be a probability space and let $f_T$,  with $T\in \mathscr{T}_{d}$,  be positive measurable functions on $X$.  We \emph{assume} that the minima of these functions are Weibull distributed,  i.e.  we suppose that
there exist $a,c >0$, $b\ge 0$ such that
given any sequence of finite subsets $\Delta_n$ of $\mathscr{T}_{d}$
satisfying 
\begin{equation}
\label{eq:sparse}
|\Delta_n|\to\infty\quad\hbox{and}\quad
 \frac{\displaystyle \min_{T\ne T'\in\Delta_n} \max_{p} |\log T_p - \log T'_p|}{\log |\Delta_n|} \to \infty,
\end{equation}
the minima 
\[
\mathfrak{M}_{\Delta_n} := \min \big\{ f_{T} :\, T\in \Delta_n\big\}
\]	
satisfy
\[
|\Delta_n|^a(\log |\Delta_n|)^{b}\cdot \mathfrak{M}_{\Delta_n} \xLongrightarrow[\mu]{}
\Wei\big(c,a\big).
\]

\begin{proposition}\label{p:log1}
Under the above assumption, for every $\delta<d-1$ and $\mu$-a.e. $x\in X$,
$$
\liminf_{\|T\|_\infty \to\infty}\, (\log \|T\|_\infty )^{a\delta} f_T(x)=0.
$$
\end{proposition}

In particular, it follows from Theorem \ref{Th_M} that the Minkowski constants \eqref{eq:cB} satisfy
$$
\liminf_{\|T\|_\infty \to\infty}\, (\log \|T\|_\infty )^{\frac{d-1}{d}-\varepsilon} c_{B_T}(\Lambda)=0 \quad \hbox{for all $\varepsilon>0$ and $\mu_d$-a.e. $\Lambda\in X_d$}
$$
 (cf. \eqref{eq:Mink}). Also from Theorem \ref{Th_PLF} we obtain that
$$
\liminf_{\|T\|_\infty \to\infty}\, (\log \|T\|_\infty )^{d-1-\varepsilon} m_{B_T}(\Lambda)=0 \quad \hbox{for all $\varepsilon>0$ and $\mu_d$-a.e. $\Lambda\in X_d$}.
$$
In particular, there exist $v_i\in \Lambda$ such that 
$$
\lim_{i\to\infty}\, (\log \|v_i\|_\infty )^{d-1-\varepsilon} \Pi(v_i)=0
$$
(cf. \eqref{eq:prod}).

\begin{proof}[Proof of Proposition \ref{p:log1}]
For the sake of a contradiction,  we assume that there exist $c_1,c_2>0$ such that
$$
\mu\big( \{x\in X:\, {\liminf}_{\|T\|_\infty \to\infty}\, (\log \|T\|_\infty)^{a\delta} f_T(x) > c_1  \}\big)> c_2.
$$
For a finite subset $\Delta \subset \mathscr{T}_{d}$, we write 
\begin{align*}
\rho(\Delta)  := \max \{ (\log \|T\|_\infty )^{a\delta}:\, T\in\Delta \}\quad\hbox{and}\quad
\mathfrak{M}_\Delta(x)  := \min \{ f_T(x):\, T\in\Delta \}.
\end{align*}
We observe that 
$$
\liminf_{\|T\|_\infty \to\infty}\, (\log \|T\|_\infty)^{a\delta} f_T(x) > c_1
$$
means that there exists $T_0(x)$ such that 
$$
(\log \|T\|_\infty)^{a\delta} f_T(x) > c_1\quad\hbox{when $\|T\|_\infty\ge T_0(x)$.}
$$
In particular, when $\|T\|_\infty\ge T_0(x)$ for all $T\in \Delta$,
$$
\rho(\Delta) \mathfrak{M}_\Delta(x)\ge \min\{(\log \|T\|_\infty)^{a\delta} f_T(x):\, T\in \Delta \}>c_1.
$$
Let us pick a sequence of finite subsets $\Delta_n$ such that 
\begin{equation} \label{eq:cond2}
\min\{\|T\|_\infty:\,  T\in \Delta_n\}\to \infty\quad\hbox{ as $n\to\infty$.}
\end{equation}
Then the set
$$
\{x\in X:\, \exists T_0(x):\, (\log \|T\|_\infty)^{a\delta} f_T(x) > c_1 \hbox{ when $\|T\|_\infty\ge T_0(x)$} \}
$$
is contained in 
$$
\{x\in X:\, \exists n_0(x):\, \rho(\Delta_n) \mathfrak{M}_{\Delta_n} (x) > c_1 \hbox{ for all $n\ge n_0(x)$} \}.
$$
In particular, we conclude that 
$$
\mu\big(\{x\in X:\, \exists n_0(x):\, \rho(\Delta_n) \mathfrak{M}_{\Delta_n} (x) > c_1 \hbox{ for all $n\ge n_0(x)$} \}\big) >c_2.
$$
Setting 
$$
\Omega_n:=\{x\in X:\, \rho(\Delta_n) \mathfrak{M}_{\Delta_n} (x) > c_1\},
$$
we conclude that $\mu(\liminf \Omega_n)>c_2$. This implies that
\begin{equation}
\label{eq:contr}
\mu(\Omega_n)>c_2\quad\hbox{ for all sufficiently large $n$.}
\end{equation}
Let us for now additionally assume that the sets $\Delta_n$ satisfy \eqref{eq:sparse} and 
\begin{equation} \label{eq:cond3}
\frac{|\Delta_n|^a (\log|\Delta_n|)^b}{\rho(\Delta_n)}\to \infty.
\end{equation}
Below we show that such sequences of sets really do exist.  \\

We choose $u\in \bR$ so that $e^{-(\lambda u)^a}<c_2$.
Since $c_1\frac{|\Delta_n|^a (\log|\Delta_n|)^b}{\rho(\Delta_n)}>u$ 
for sufficiently large $n$, we deduce that  
$$
\mu(\Omega_n)\le 
\mu\big(\{x\in X:\, |\Delta_n|^a (\log|\Delta_n|)^b\, \mathfrak{M}_{\Delta_n} (x) > u \}\big)
\to e^{-(\lambda u)^a}<c_2.
$$
However, this contradicts \eqref{eq:contr}. Therefore, we conclude that 
for all $c_1>0$,
$$
\mu\big( \{x\in X:\, {\liminf}_{\|T\|_\infty \to\infty}\, (\log \|T\|_\infty)^{a\delta} f_T(x) > c_1\}\big)= 0.
$$
This implies the proposition,  provided that we can show that we can construct $(\Delta_n)$ satisfying
 the conditions \eqref{eq:sparse}, \eqref{eq:cond2}, \eqref{eq:cond3}.  For this we use parameters $\ell_n,\theta_n\to \infty$. Let
$$
\Delta_n:=\big\{T(s_1,\ldots, s_{d-1}):\, 1\le s_1,\ldots,s_{d-1}\le \ell_n\big\},
$$
where
$$
T(s_1,\ldots, s_{d-1}):=\big(\theta_n^{s_1},\ldots,\theta_n^{s_{d-1}}, \theta_n^{-\sum_k s_k} \big).
$$
Then $|\Delta_n|=\ell_n^{d-1}\to \infty$, and the condition \eqref{eq:sparse} is equivalent
 to $\frac{\log \theta_n}{\log \ell_n}\to \infty$, so that 
$\theta_n=\ell_n^{\omega_n}$ with $\omega_n\to\infty$.
The condition \eqref{eq:cond2} is obviously satisfied,
and the condition \eqref{eq:cond3} holds when 
$$
\frac{\ell_n^{a(d-1)}}{\log(\theta_n^{\ell_n})^{a\delta}}=\frac{\ell_n^{a(d-1)- a\delta }}{(\omega_n\log\ell_n)^{a\delta}} \to\infty.
$$
Since $\delta<d-1$, this can be arranged by taking for instance $\omega_n=\log \ell_n$.
\end{proof}	

An analogue of Proposition \ref{p:log1} also holds in the setting of Theorems \ref{Th_Dirichlet} and \ref{Th_Gal}.
We fix a tuple $\bd = (d_1,d_2)$
such that $d = d_1 + d_2$, $d_1,d_2\ge 1$ and consider the sub-semigroup $\mathscr{T}^{+}_{\bd}$ of $\mathscr{T}_{d}$ defined in \eqref{eq:TT}.
Let $\mathscr{C}$ be a cone in $\mathscr{T}_{\bd}^{+}$ with finitely many faces. 
We consider positive measurable functions $f_T$, $T\in \mathscr{C}$, on $X$ 
and assume that the minima of these functions satisfy the Weibull distribution. Namely, we suppose that
there exist $a>0$, $b\ge 0$ such that
given any sequence of finite subsets $\Delta_n$ of $\mathscr{C}$
satisfying $|\Delta_n|\to\infty$ and 
\begin{equation}
\label{eq:sparse2}
 \frac{\displaystyle \min_{T\ne T'\in\Delta_n} \max_{p} |\log T_p - \log T'_p|}{\log |\Delta_n|} \to \infty
\quad \hbox{and} \quad
 \frac{\displaystyle \min_{T\in\Delta_n} \lfloor T\rfloor_{\bd}}{\log |\Delta_n|} \to \infty
\end{equation}
the minima 
\[
\mathfrak{M}_{\Delta_n} := \min \big\{ f_{T} :\, T\in \Delta_n\big\}
\]	
satisfy
\[
|\Delta_n|^a(\log |\Delta_n|)^{b}\cdot \mathfrak{M}_{\Delta_n} \xLongrightarrow[\mu]{}
\Wei\big(c,a\big).
\]

\begin{proposition}\label{p:log2}
	Under the above assumption, if $d_0=\dim (\mathscr{C})$,
	for every $\delta<d_0$ and $\mu$-a.e. $x\in X$,
	$$
	\liminf_{\|T\|_\infty \to\infty}\, (\log \|T\|_\infty )^{a\delta} f_T(x)=0.
	$$
\end{proposition}

Let us apply this proposition in the setting of Theorem \ref{Th_Dirichlet} with 
$d_1=d-1$ and $d_2=1$ with the cone 
$$
\mathcal{C}=\{(t^{1/(d-1)},\ldots, t^{1/(d-1)}, t^{-1}):\, t\ge 1 \}.
$$
We conclude that 
$$
\liminf_{\|T\|_\infty \to\infty}\, (\log \|T\|_\infty )^{\frac{1}{d-1}-\epsilon} k_T(\alpha)=0.
$$
for every $\varepsilon>0$ and Lebesgue almost every $\alpha\in [0,1]^{d-1}$ (cf. \eqref{DS}).

In the setting of Theorem \ref{Th_Gal} with $d_1=d-1$ and $d_2=1$ with the cone
$$
\mathcal{C}=\{(t^{1/(d-1)}s_1,\ldots, t^{1/(d-1)}s_{d-1}, t^{-1}):\,\, s_1,\ldots, s_{d-1}\ge t^{-1/(d-1)},\, t\ge 1, \, s_1\cdots s_{d-1}=1 \},
$$
we deduce that 
$$
\liminf_{\|T\|_\infty \to\infty}\, (\log \|T\|_\infty )^{d-1-\epsilon} g_T(\alpha)=0.
$$
for every $\varepsilon>0$ and Lebesgue almost every $\alpha\in [0,1]^{d-1}$
(cf.  \eqref{eq:Gal}).

\begin{proof}[Proof of Proposition \ref{p:log2}]
The proof proceeds as in Proposition \ref{p:log1}. We just have to justify existence of a sequence of finite subsets $\Delta_n$  of $\mathcal{C}$ satisfying the required conditions. We use parameters $\ell_n,\theta_n\to \infty$. Since $\dim(\mathcal{C})=d_0$, there exist (multiplicatively) independent 
$$
A_{n,k}=(\theta_n^{w_{1,k}},\ldots,\theta_n^{w_{d,k}}), \quad k=1,\ldots, d_0,
$$
elements in $\mathcal{C}$ (here $w_{l,k}>0$ for $l\le d_1$ and $w_{l,k}>0$ for $l> d_1$).
Let
$$
\Delta_n:=\big\{T(s_1,\ldots, s_{d_0}):\, 1\le s_1,\ldots,s_{d_0}\le \ell_n\big\},
$$
where
$$
T(s_1,\ldots, s_{d_0}):=A_{n,1}^{s_1}\cdots A_{n,d_0}^{s_{d_0}}.
$$
One can check as in the proof of Proposition \ref{p:log1} that when
$\theta_n=\log \ell_n$, the sets $\Delta_n$ satisfy the required conditions.
We note that in this case we also need an additional condition in \eqref{eq:sparse2},
but this condition also follows from $\frac{\log \theta_n}{\log \ell_n}\to \infty$.
\end{proof}

\end{document}